\documentclass[leqno,12pt]{article}

\newcommand\blfootnote[1]{%
  \begingroup
  \renewcommand\thefootnote{}\footnote{#1}%
  \addtocounter{footnote}{-1}%
  \endgroup
}
\def\today{${\scriptscriptstyle\number\day-\number\month-\number\year}$}
\usepackage{graphicx}
\usepackage{fancyhdr}
\usepackage{amsmath}
\usepackage{amsthm}
\usepackage{amsfonts}
\newcommand{\eqlb}[1]{\begin{equation}\label{#1}}
\newcommand{\eqe}{\end{equation}}

\newtheorem{theorem}{Theorem}[section]
\newtheorem{lemma}[theorem]{Lemma}
\theoremstyle{definition}
\newtheorem{definition}[theorem]{Definition}
\newtheorem{remark}[theorem]{Remark}

\newcommand{\rf}[1]{(\ref{#1})}

\newcommand{\lp}[3]{|#1|_{#2}^{#3}}
\newcommand{\hp}[3]{\|#1\|_{#2}^{#3}}
\newcommand{\fr}[2]{{#1 \over #2}}
\newcommand{\vph}[1]{v_{\varphi #1 }}

\renewcommand{\d}{\mathrm{d}}
\date{}
\def\m@th{\mathsurround=0pt}
\def\eqal#1{\null\,\vcenter{\openup\jot\m@th
\ialign{\strut\hfil$\displaystyle{##}$&&$\displaystyle{{}##}$\hfil
 \crcr#1\crcr}}\,}
\def\matrix#1{\null\,\vcenter{\normalbaselines\m@th
 \ialign{\hfil$##$\hfil&&\quad\hfil$##$\hfil\crcr
 \mathstrut\crcr\noalign{\kern-\baselineskip}
 #1\crcr\mathstrut\crcr\noalign{\kern-\baselineskip}}}\,}

\def\N{{\Bbb N}}
\def\R{{\Bbb R}}
\def\al{{\alpha}}
\def\dl{{\delta}}
\def\O{{\Omega}}
\def\o{{\omega}}
\def\sg{{\sigma}}
\def\G{{\Gamma}}
\def\Ot{{\Omega^t}}
\def\t{{\theta}}
\def\io{{\intop_\Omega}}
\def\iot{\intop_{\Omega^t}}
\def\dt{\frac{d}{dt}}
\def\ve{\varepsilon}
\def\vp{\varphi}
\def\ps{{\psi_1}}
\def\vphi{{v_\varphi}}
\def\fphi{{f_\varphi}}
\def \ll{\left}
\def\rr{\right}
\def\nb{\nabla}
\def\iy{\infty}
\def\pa{\partial}

\def\divv{{\rm div}}
\def\rot{{\rm rot}}
\def\swirl{{\rm swirl}}
\def\cd{\cdot}
\def\lb{\label}
\def\les{\leqslant}

\numberwithin{equation}{section}
\def\bye{\end{document}}
\title{The global estimate for regular
axially-symmetric solutions to the 
\linebreak Navier Stokes equations coupled
with \linebreak the heat conduction}
\author{Wies\l aw J. Grygierzec, Wojciech M. Zaj\c{a}czkowski}

\begin{document}
\input amssym.def
\input amssym.tex
\maketitle
\thispagestyle{fancy}

\blfootnote{
\noindent WG: Department of Statistics and Social Policy, University of Agriculture in Kraków, Al. Mickiewicza 21, 31-120 Kraków, Poland, email: wieslaw.grygierzec@urk.edu.pl}
\blfootnote{\noindent WZ: Institute of Mathematics (emeritus professor), Polish Academy of Sciences, \'Sniadeckich 8, 00-656 Warsaw, Poland, and Institute of Mathematics and Cryptology, Cybernetics Faculty, Military University of Technology, S. Kaliskiego 2, 00-908 Warsaw, Poland, e-mail: wz@impan.pl}

\begin{abstract}
The axially-symmetric solutions to the Navier-Stokes equations coupled with the heat conduction are considered in a bounded cylinder $\Omega \subset \mathbb{R}^3$. We assume that $v_r, \vphi, \omega_\vp$ vanish on the lateral part $S_1$ of the boundary $\pa \Omega$ and $v_z, \omega_{\varphi}, \partial_z v_{\varphi}$ vanish on the top and bottom of the cylinder, where we used standard cylindrical coordinates and $\omega=\rot v$ is the vorticity of the fluid. Moreover, vanishing of the heat flux through the boundary is imposed.\\
Assuming existence of a sufficiently regular solution we derive a global a priori estimate in terms of data. The estimate is such that a global regular solutions can be proved. We prove the estimate because some reduction of nonlinearity are found. Moreover, we need that $f(p)\equiv \hp{\vphi}{L_t^\iy L_x^p }{}/\hp{\vphi}{L_t^\iy L_x^\iy }{}$ is bounded from below by a positive constant. The quantity $f(p)$ is close to 1 for large $p$ because $f(\iy)=1$.
Moreover, deriving the global estimate for a local solution implies a possibility of its extension in time as long as the estimate holds.

\end{abstract}

\noindent
\noindent
{\bf MSC:} 35A01, 35B01, 35B65, 35Q30, 76D03, 76D05\\
Key words: Navier-Stokes equations, heat conduction, axially-symmetric solutions, cylindrical domain, existence of global regular solutions

\section{Introduction}\label{s1}

We are concerned with the 3d incompressible axially-symmetric Navier-Stokes equations coupled with the heat conduction. We derive a global a priori estimate for regular axially-symmetric solutions to the system  in  a cylindrical domain  
\begin{equation}
\begin{split}
\partial_{t}v+v\cdot\nabla v-\nu\Delta v+\nabla p&=\alpha(\theta)f, \\
\divv v&=0\quad {\rm in   } \; \Omega^T,
\end{split}
\label{1.1}
\end{equation}
and 
\begin{equation}\label{1.2}
    \partial_{t}\theta+v\cdot\nabla \theta-\kappa\Delta \theta=g\quad {\rm in   } \; \Omega^T,
\end{equation} 
where $\Omega^T=\Omega \times (0,T),\, T>0, \,v(x,t)\in\R^3$ denotes the velocity field, 
$p=p(x,t)\in\R$ denotes  the pressure function, $\theta=\theta(x,t)\in\R_+$ denotes the temperature, $f=f(x,t)\in\R^3$ denotes  the external force field, $g=g(x,t)$ denotes the heat sources, $\nu>0$ is the constant viscosity coefficient and $\kappa>0$ denotes the constant heat conductivity.
By $\Omega\subset\R^3$ we assume a finite cylinder
$$
\Omega=\{x\in\R^3\colon x_1^2+x_2^2<R^2,|x_3|<a\},
$$
and $a$, $R$ are given positive constants and  $x=(x_1,x_2,x_3)$ are Cartesian coordinates.
We note that
\[S:=\partial \Omega =S_1 \cup S_2,
\]
where
\[\begin{split}
&S_1=\{x\in\R^3\colon\sqrt{x_1^2+x_2^2}=R,x_3\in(-a,a)\},\\
&S_2=\{x\in\R^3\colon\sqrt{x_1^2+x^2}<R,x_3\in\{-a,a\}\},\\
\end{split}
\]
denote the lateral boundary and the top and bottom parts of the boundary, respectively.

In order to state the boundary conditions stating our main result we describe our problem in cylindrical coordinates $r,\varphi, z$ defined by 
\[
x_1=r\cos\varphi,\quad x_2=r\sin\varphi,\quad x_3=z.
\]
and we will use standard cylindrical unit vectors, so that, for example, \[ v=v_r\bar e_r+v_\varphi\bar e_\varphi+v_z\bar e_z \] where
$\bar e_r=(\cos\varphi,\sin\varphi,0),\ \ \bar e_\varphi=(-\sin\varphi,\cos\varphi,0),\ \ \bar e_z=(0,0,1)$.
We will denote partial derivatives  by using by using subscript comma notation, e.g.
\[v_{r,z}:=\partial_z v_r
\]
We assume the boundary conditions
\begin{equation}\label{1.3}
    \begin{aligned}
        &v_r =v_{\varphi}=\omega_{\varphi} =0 &{\rm   on   }&\;S_1 ^T=S_1\times(0,T),\\
        &v_z=\omega_{\varphi}=v_{\varphi ,z}=0 &{\rm on }&\;S_2^T=S_2 \times(0,T),\\
       &\bar n \cdot \nabla \theta =0 &{\rm   on   }&\;S^T=S\times(0,T)
   \end{aligned}
\end{equation}
where $\omega=\mathrm{curl}$ denotes the vorticity vector and we assume initial conditions
\begin{equation}\label{1.4}
    \begin{split}
        &v|_{t=0}=v(0)=v_0,\\
       &\t|_{t=0}= \theta(0)=\Theta_0,\\
    \end{split}
\end{equation}
where $v_0$ is given divergence free vector and $v_0, \theta _0$ satisfy boundary conditions \rf{1.3} These boundary conditions have appeared in the work of \linebreak Ladyzhenskaya\cite{L}.

We will denote the swirl by 
\[
u:=rv_{\varphi}.
\]
Note that 
\begin{equation}\begin{split} 
&\omega_r=-v_{\varphi,z}=-\fr{1}{r}u_{,z},\\
&\omega_\varphi=v_{r,z}-v_{z,r},\\
&\omega_z=\fr{1}{r}(rv_\varphi)_{,r}=v_{\varphi,r}+\fr{v_\varphi}{ r}=\fr{1}{r}u_{,r}.\\
\end{split}
\label{1.5}
\end{equation}
so that the boundary conditions \rf{1.3} imply 
\begin{equation}\label{1.6}
    \begin{aligned}
&\omega_r=v_{z,r}=u=0,\omega_z =v_{\varphi,r}&{\rm on}\; S_1^T,\\
&\omega_r=v_{r,z}=\omega_{z,z}=u_{,z}=0 &{\rm on}\;S_2^T
\end{aligned}
\end{equation}
The Navier-Stokes equations (\ref{1.1})  in cylindrical coordinates take the form 
\begin{equation}\begin{split}
v_{r,t}+v\cdot\nabla v_r-\fr{v_\varphi^2}{ r}-\nu\Delta v_r+\nu\fr{v_r} {r^2}&=-p_{,r}+\alpha(\theta)f_r,\\
v_{\varphi,t}+v\cdot\nabla v_\varphi+\fr{v_r}{r}v_\varphi-\nu\Delta v_\varphi+\nu\fr{v_\varphi}{r^2}&=\alpha(\theta)f_\varphi,\\
v_{z,t}+v\cdot\nabla v_z-\nu\Delta v_z&=-p_{,z}+\alpha(\theta)f_z,\\
(rv_r)_{,r}+(rv_z)_{,z}&=0\\
\end{split}
\label{1.7}
\end{equation}
where 
\begin{equation*}
v\cdot\nabla=(v_r\bar e_r+v_z\bar e_z)\cdot\nabla=v_r\partial_r+v_z\partial_z,
\;\Delta u=\fr{1}{r}(ru_{,r})_{,r}+u_{,zz}.
\end{equation*}
Using that $\o_r=-\vph{,z},\; \o_\vp=v_{r,z}-v_{z,r},\; \o_z=\fr1 r (r\vphi)_{,r}$
the  vorticity formulation becomes
\begin{equation}\label{1.8}
\begin{split}
&\omega_{r,t}+v\cdot\nabla\omega_r-\nu\Delta\omega_r+\nu{\omega_r\over r^2}=\omega_rv_{r,r}+\omega_zv_{r,z}-
\dot\alpha\theta_{,z}f_{\varphi}+\alpha F_r\\
&\omega_{\varphi,t}+v\cdot\nabla\omega_\varphi-{v_r\over r}\omega_\varphi-\nu\Delta\omega_\varphi+\nu{\omega_\varphi\over r^2}={2\over r}v_\varphi v_{\varphi,z}+\dot\alpha(\theta_{,z}f_r-\t_{,r}f_z)\\
& \hspace{10cm}+\alpha F_\varphi,\\
&\omega_{z,t}+v\cdot\nabla\omega_z-\nu\Delta\omega_z= \omega_rv_{z,r}+\omega_zv_{z,z}+\dot\alpha\fr1 r(r\t)_{,r}f_\vp+\alpha F_z,\\
\end{split}
\end{equation}
where $F:=\mathrm{curl}f$ and the swirl  equation is
\begin{equation}\label{1.9}
\begin{split}
       &u_{,t}+v\cdot \nabla u- \nu \Delta u +{2\nu\over r}u_{,r}=\alpha r f_\varphi \equiv\alpha f_0\\
       &u|_{t=0}=u_0=r\vphi(0).
\end{split}
 \end{equation}
We will use the notation
\begin{equation}\label{1.10}
    (\Phi,\Gamma)=({\omega_r \over r},{\omega_{\varphi}\over r}),
\end{equation}
and we note that $\Phi,\Gamma$ satisfy 
\begin{equation}\label{1.11}
\begin{split}
    \Phi_{,t}+v\cdot\nabla\Phi-\nu\bigg(\Delta+{2\over r}\partial_r\bigg)\Phi- (\omega_r\partial_r+\omega_z\partial_z){v_r\over r} \\=\dot \alpha \theta_{,z}\fphi/r+\alpha F_r/r\equiv -\dot \alpha \theta_{,z}\bar \fphi +\alpha \bar F_r,
\end{split}
\end{equation}
\begin{equation}\label{1.12}
    \begin{split}
        \Gamma_{,t}+v\cdot\nabla\Gamma-\nu\bigg(\Delta+{2\over r}\partial_r\bigg)\Gamma&+ 2{v_\varphi\over r}\Phi\\&=
        \dot \alpha (\theta _{,z}\bar f_r- \theta_{,r}\bar f_z)+\alpha \bar F_\varphi.
      \end{split}
\end{equation}
where $\bar f=f/r,\;\bar F=F/r$ and recall  \cite{CFZ},(1.6). Moreover by (\ref{1.3}),(\ref{1.4}), $\Gamma$ and $\Phi$ satisfy the boundary and initial conditions
\begin{equation}
    \label{1.13}
    \begin{split}
&\Phi=\Gamma=0 \quad \text{ on } S^T, \\
&  \Phi|_{t=0}=\Phi(0)\equiv \fr{\o _r(0)}r,\\
& \G|_{t=0}=\G(0)\equiv \fr{\o _\vp (0)}r.
\end{split}
\end{equation}
Recall that (\ref{1.7})$_4$ implies existence of the stream function $\Psi$ which solves the problem 
\begin{equation}\label{1.14}
    \begin{split}
        -\Delta \psi + {\psi \over {r^2}}&=\omega_{\varphi},\\
        \psi|_S&=0.
    \end{split}
\end{equation}
Then $v$ can be expressed in terms of the stream function,
\begin{equation}\lb{1.15}
    \begin{aligned}
        v_r&=-\psi_{,z},&v_z&={1 \over r}\bigg(r\psi\bigg)_{,r}=\psi_{,r}+{\psi \over r},\\
        v_{r,r}&=-\psi_{,zr},& v_{,zz}&=\psi_{,rz}+{\psi_{,z}\over r},\\
        v_{r,z}&=-\psi_{,zz},&v_{z,r}&=\psi_{,rr}+{1\over r}\psi_{,r}-{\psi \over r^2}.
    \end{aligned}
\end{equation}
We will also use the modified stream function
\begin{equation}\label{1.16}
    \psi_1:={\psi \over r},
\end{equation}
which satisfies
\begin{equation}\label{1.17}
    \begin{split}
        -\Delta \psi_1 -{2 \over r }\psi_{1,r}&=\Gamma,\\
         \psi_1|_S &=0,    \\
        \end{split}
\end{equation}

and we express $v$ in terms of $\psi_1$  by
\begin{equation}\begin{aligned}
&v_r=-r\psi_{1,z}, &v_z&=(r\psi_1)_{,r}+\psi_1=r\psi_{1,r}+2\psi_1,\\
&v_{r,r}=-\psi_{1,z}-r\psi_{1,rz}, &v_{z,r}&=3\psi_{1,r}+r\psi_{1,rr},\\
&v_{r,z}=-r\psi_{1,zz}, &v_{z,z}&=r\psi_{1,rz}+2\psi_{1,z}.\\
\end{aligned}
\label{1.18}
\end{equation}
Projecting (\ref{1.17})$_1$ on $S_2 $ gives $\psi_{1,zz}=-\Gamma$ on $S_2$ and recalling that $\Gamma|_{S_2}=0$ by \rf{1.2} we obtain
\begin{equation}\label{1.19}
    \psi_{1,zz}=0\quad \text{on  }S_2.
\end{equation}
We have to emphasize that all estimates in this paper are derived for regular solutions. This meas that that smooth $v$ and $\psi$ admits the following expansions near the axis
\begin{gather}
    v_r(r,z,t)=a_1(z,t)r+a_2(z,t)r^3+\ldots,\label{1.20}\\
     v_\varphi(r,z,t)=b_1(z,t)r+b_2(z,t)r^3+\ldots,\label{1.21}\\
       \psi(r,z,t)=d_1(z,t)r+d_2(z,t)r^3+\ldots,\label{1.22}
\end{gather}
   In particular
\begin{align}
&\psi_1(r,z,t)=d_1(z,t)+d_2(z,t)r^2+\ldots ,\label{1.23}\\
&\psi_{1,r}(r,z,t)=2d_2(z,t)r+\ldots ,\label{1.24}
\end{align}
which was shown by Liu\&Wang \cite{LW}. To show (\ref{1.20})-(\ref{1.24}) it suffices that $v,\psi\in W_2^{3,3/2}(\Omega^T)$. We show in Section 7 that this is true as long as the quantity 
\begin{equation}
    X(t):=\|\Phi\|_{V(\Omega^t)}+\|\Gamma\|_{V(\Omega^t)},
\end{equation}
where  
\[\|w\|_{V(\Omega^t)}:=|w|_{2,\infty,\Omega^t}+|\nabla w|_{2,\Omega^t},\]
remains bounded.

\begin{theorem}\lb{th1.1} (a priori estimate)
\begin{enumerate}
\item Suppose that $v, \theta$ is a smooth solution to problem (1.1)-(1.4).
\item Suppose that quantities $D_{0}, \ldots, D_{12}, B_{1}$, defined in Section 2.4 , are finite for any $t \in \mathbb{R}_{+}$.
\item Suppose that there exists a positive constant $c_{0}$ such that
$$
 \fr{\lp{\vphi}{d,\iy,\Ot}{}}{\lp{\vphi}{\iy,\Ot}{}}\geqslant c_{0}
$$

\end{enumerate}

for $d \geqslant 3$.

Then there exists an increasing positive function $\phi$ such that 
\begin{equation}\lb{1.26}
    X(t) \leq \phi\left(D_{1}, \cdots, D_{12}, B_{1}\right)
\end{equation}
\end{theorem}
\begin{proof}
In Lemma \ref{2.2} is proved the existence of positive constants $\theta_{*}, \theta^{*}, $\linebreak $\theta_{*}<\theta^{*}$ such that $\theta_{*} \leq \theta(t) \leq \theta^{*}$ for any $t \in \mathbb{R}_{+}$.  In Lemma \ref{2.3} the following energy estimate 

\[
\|\t\|_{V(\Ot)} \leq D_{0},\; t \in \mathbb{R}_{+}
\]
is proved.  Lemma \ref{lemma 2.4} yields the energy estimate for velocity $v$
\[
\hp{v}{V(\Ot)}{2} +\nu\iot\bigg({v_r^2 \over r^2}+{v_{\vp}^2 \over r^2}\bigg )dx dt' \leq D_{1}^{2} .
\]
The maximum principle for swirl $u=r\vphi$ is proved in Lemma \ref{lemma 2.5}
$$
|u|_{\infty, \Ot} \leq D_{2} .
$$
Lemma \ref{lemma 4.1} and Remark \ref{rm 4.2} imply
\begin{equation}\lb{1.27}
    \begin{aligned}
X^{2}(t)& \leqslant \phi_{1}\left(\theta_*, \theta^*, D_{1}, D_{2}, B_{1}, D_{3}, R\right) \\
& \cdot \ll[(1+\lp{\vphi}{\iy ,\Ot}{\ve_0})\lp{\vphi}{d, \iy ,\Ot}{\ve}\lp{\Phi}{2,\Ot}{\t _0} X^{2-\t_0}+1\rr],
\end{aligned}
\end{equation}

where 
\begin{align*}
    &\theta_{0}=\left(1-\frac{3}{d}\right) \varepsilon_{1}-\frac{3}{d} \varepsilon_{2},  
    &&1+\frac{\varepsilon_{2}}{\varepsilon_{1}}<\frac{d}{3},\\
    &\varepsilon_{1}\left(1-\frac{3}{d}\right)<1+\frac{3}{d} \varepsilon_{2}, 
   && d>3 
\end{align*}
and $\varepsilon_{0}$ is as small as we need.
Lemma \ref{lemma 6.1} gives
\begin{equation}\lb{1.28}
    |\Phi|_{2, \Ot}^{2} \leqslant \phi_{2}\left(D_{1}, D_{2}, D_{4}, D_{5}\right)\left(1+\lp{\vphi}{\iy,\Ot}{2\ve_0} \right) X\\
+\phi_{3}\left(D_{0,} D_{10}\right).
\end{equation}

To prove the inequality we need $\mathrm{H}^{2}-\mathrm{H}^{3}$ estimates for the modified stream function $\psi_{1}$ proved in Section 3. To prove the inequalities in Section 3 we need Liu-Wang expansions \rf{1.20}-\rf{1.22}. Moreover, we need also the energy estimates for $\nb u$ proved in Section 5 in the form
$$
\|\nb u\|_{V(\Ot)} \leqslant \phi\left(D_{4}, D_{5} \right) .
$$
Lemma \ref{lemma 6.2} implies
\begin{equation}\lb{1.29}
    \lp{\vphi}{\iy,\Ot}{}\le \Phi_4(D_1,D_2)X^{3/4}+D_{11}.
\end{equation}
Finally, Lemma \ref{lemma 6.3} yields
\begin{equation}\lb{1.30}
    \lp{\vphi}{d,\iy,\Ot}{}  \leq \phi_{5}(D_{1}, D_{2}, c_{0}, D_{12}),
\end{equation}
where conditions \rf{6.20} are used.
To prove \rf{1.30} we need  Assumption 3. Using \rf{1.28}-\rf{1.30}  in \rf{1.27} yields
\begin{equation}\lb{1.31}
    X^2(t)\leq \phi(D_0,\ldots D_{12},B_1,R,C_0)\ll[  X^{3\ve_0+2-\t_0/2}+1 \rr].
\end{equation}
Since
$$
3 \varepsilon_{0}+2-\fr{\theta_0} 2<2
$$
estimate \rf{1.26} holds. This ends the proof.
\end{proof}
\begin{remark}\label{rm1.2}
 We note that condition from Assumption 3 from Theorem 1.1 can be justified assuming sufficient regularity of $v_{\varphi}$ on (0,T). Namely, let
$$
f(x, t)=\frac{\left|v_{\varphi}(x, t)\right|}{\lp{\vphi(t)}{\iy,\O}{}} .
$$
Then $|f|_{\iy, \Omega}=1$ for any $t \in(0, T)$.
Suppose that $f \in C^{\alpha, \alpha / 2}\left(\Omega^{t}\right)$ for some $\alpha \in(0,1)$. Then for every $\varepsilon>0$ there exists a set $A \subset \Omega$ of positive measure $|A|$ such that
$$
f(x, t) \geqslant 1-\varepsilon \text { for } x \in A,\; t \in(0, \mathrm{~T}) \text {. }
$$

This gives $|f|_{d, \Omega}^{d} \geqslant \int_{A}|f(x, t)|^{d} d x \geqslant|A|(1-\varepsilon)^{d}$. Hence 
$$|f|_{d, \Omega} \geqslant|A|^{1 / d}(1-\varepsilon)$$
and
$$
\sup _{t}|f|_{d, \Omega} \geqslant|f|_{d, \Omega} \geqslant|A|^{1 / d}(1-\varepsilon),
$$
from which the Assumption 3 of Theorem \ref{th1.1} holds.
\end{remark}
\begin{theorem}\label{th1.3}
Suppose that the assumptions of Theorem 1.1 holds. Suppose that $f, g \in W_{2}^{2,1} (\Ot),\; v_{0}, \theta_{0} \in H^{3}(\O)$.

Then for smooth solutions to problem \rf{1.1} - \rf{1.4} the following estimate holds

\begin{multline}\lb{1.32}
 \|v\|_{W_{2}^{4,2}\left(\Omega^{t}\right)}+\|\theta\|_{W_{2}^{4,2}\left(\Omega^{t}\right)} \\
 \leq \phi(D_{0}, \cdots, D_{12}, B_{1},\|f\|_{W_{2}^{2,1} (\Ot)},\|g\|_{W_{2}^{2,1} (\Ot )} ,\|v_0\|_{H^3(\Omega)},\|\t_0\|_{H^3(\O)}).
\end{multline}
\end{theorem}
\begin{proof}
    (see Section 7).
\end{proof}

\section{Preliminaries}\label{s2}
\subsection{Notations}
We use the following notations for the Lebesgue and Sobolev spaces
\[
\begin{split}
  &  \|u\|_{L_p(\Omega)}=|u|_{p,\Omega},\quad\|u\|_{L_p(\Omega^t)}=|u|_{p,\Omega^t},\\
    & \|u\|_{L_{p,q}(\Ot)}=\|u\|_{L_q(0,t;L_p(\Omega)}=|u|_{p,q,\Omega^t},\;p,q\in[1,\infty],
\end{split}
\]
Let $W^s_p(\Omega), s\in \N, \Omega\subset \R^3$ be the Sobolev space with the finite norm 
\[
\|u\|_{W^s_p(\Omega)}=\bigg(\sum _{|\alpha|\le s}\intop _\Omega |D^\alpha _x u(x)|^p dx\bigg)^{1/p},
\]
where 
\[\begin{split}
&D_x^\alpha =\partial^{\alpha_1}_{x_1}\,\partial^{\alpha_2}_{x_2}\,\partial^{\alpha_3}_{x_3},\;|\al|=\al_1+\al_2+\al_3, \\ 
&\alpha_i\in\N_0=\N\cup\{0\}, i=1,2,3,\;p\in[1,\infty] 
\end{split}
\]
We set $H^s(\Omega)=W^s_2(\Omega)$ and
\[
\begin{aligned}
    \|u\|_{H^s(\Omega)}&=\|u\|_{s,\Omega},&\|u\|_{W^s_p(\Omega)}&=\|u\|_{s,p,\Omega},\\
    \|u\|_{L_q(0,t;W^s_p(\Omega)}&=\|u\|_{s,p,q,\Omega^t},&\|u\|_{s,p,p,\Omega^t}&=\|u\|_{s,p,\Omega ^t},
    \end{aligned}
\]
where $s\in\N\cup\{0\},\;p,q\in[1,\infty]$.
We need the energy type space $V(\Omega^t)$ appropriate for description of weak solutions to the Navier-Stokes equations and the heat equation
\[
\|u\|_{V(\Omega^t)}=|u|_{2,\infty,\Omega_t}+|\nabla u|_{2,\Omega_t}.
\]
\subsection{Basic estimates}
\begin{lemma}\label{lemma 2.1}
    Let $\theta$ be a solution to (\ref{1.2}). Assume that there exist a positive constants $\theta_*$ , $0<\theta_* \le \theta(0)$. Assume also that $g\ge0$.
    Then solutions to (\ref{1.2}) satisfy
    \begin{equation}\label{2.1}
        \theta(t)\ge \theta_*
    \end{equation}
\end{lemma}
\begin{proof}
    Multipply (\ref{1.2}) by $(\theta-\theta_*)_-=\min\{\theta-\theta_*,0\}$ and integrate over $\Omega$. Using boundary conditions yields 
    \begin{equation}\label{2.2}
        \begin{split}
            \fr12\dt &\intop_\Omega(\theta-\theta_*)_-^2dx+\kappa \intop_\Omega |\nabla (\theta-\theta_*)|^2dx \\
            &=-\intop_\Omega v\cdot \nabla (\theta-\theta_*)_-(\theta-\theta_*)_-dx+\intop_\Omega g(\theta-\theta_*)_-dx.
        \end{split}
    \end{equation}  
    The first term on the r.h.s. vanishes because it equals
    \[
    -{1 \over 2}\intop_\Omega v\cdot \nabla (\theta-\theta_*)_-^2 dx=-{1 \over 2}\intop_Sv\cdot \bar n(\theta-\theta_*)_-^2dS=0,
    \]
    where $v\cdot \bar n|_S=0$. Since $g\ge 0$ the last term on the r.h.s. of (\ref{2.2}) is less than or equal to zero.
    Then (\ref{2.2}) yields 
    \[
    {d \over dt}|(\theta-\theta*)_-|^2_{2,\Omega}\le 0
    \]
    so
    \[
    |(\theta(t)-\theta*)_-|_{2,\Omega}\le|(\theta(0)-\theta*)_-|_{2,\Omega}
    \]
    Hence, $\theta(0)\ge\theta_*$ implies (\ref{2.1}). This ends the proof.
\end{proof}
\begin{lemma}\label{lemma 2.2}
    Let $\theta$ be a solution to (1.2), Let $\theta_* $ positive constant such that  $\theta_*\le\theta(0), $ and let $g \ge 0,\,\divv=0,\,v\cdot \bar n|_S=0$. 

Then  there exists $\t^*\ge 0$ , such that solutions to (\ref{1.2}) satisfy
    \begin{equation}\label{2.3}
        \theta_*\le\theta(t)\le\theta^*.
    \end{equation}
    and $ \t ^*=|g|_{\infty,1,\Omega}+|\theta(0)|_{\infty,\Omega} .$
\end{lemma}
\begin{proof}
    Multiply (\ref{1.2}) by $\theta^{s-1}$ and  integrate over $\Omega.$ Then we obtain
\begin{equation}\label{2.4}
    \begin{split}
        {1 \over s} {d \over dt}|\theta|^s_{s,\Omega}+{ 4\kappa(s-1) \over s^2}\io  |\nb \t ^{s/2}|^2dx&=-{ 1 \over s}\intop_\Omega v\cdot \nabla \theta^s dx\\
        &+\intop_\Omega g\theta^{s-1}dx.
    \end{split}
\end{equation}
Using that $v\cdot\bar n|_S=0$ we derive the inequality
\[     {1 \over s} {d \over dt}|\theta|^s_{s,\Omega}\le |g|_{s,\Omega}|\theta|^{s-1}_{s,\Omega}.
\]
Simplifying, we get 
\[
 {d \over dt}|\theta|_{s,\Omega}\le |g|_{s,\Omega}.
\]
Integrating with respect to time and passing with s to infinity yields
\[
|\theta(t)|_{\infty,\Omega}\le |g|_{\infty,1,\Omega}+|\theta(0)|_{\infty,\Omega}\equiv \t^*.
\]
Hence one side of (\ref{2.3}) is proved.

Multiply (\ref{1.2}) by $\theta ^{-s},\,s>0$ and integrate over $\Omega$. Then we have 
\begin{equation}\label{2.5}
    \begin{split}
        -{1 \over s-1}{d \over dt}&\intop_\Omega {1 \over \theta ^{s-1}}dx-{4s \over (s-1)^2}\intop_\Omega |\nabla  {1 \over \theta ^{(s-1)/2}}|^2dx\\
        &={1 \over s-1}\intop_\Omega v\cdot \nabla {1 \over \theta ^{s-1}}dx+\io g\t^{-s}dx.
    \end{split}
\end{equation}
Multiplying (\ref{2.5}) by $-(s-1) $ yields
\begin{equation}\label{2.6}
    \begin{split}
       {d \over dt}&\intop_\Omega {1 \over \theta ^{s-1}}dx+{4s \over (s-1)}\intop_\Omega |\nabla  {1 \over \theta ^{(s-1)/2}}|^2dx\\
        &=-\intop_\Omega v\cdot \nabla {1 \over \theta ^{s-1}}dx-(s-1)\io g\t^{-s}dx.
    \end{split}
\end{equation}
In review of boundary condition $v\cdot \bar n|_S=0$ the first term on the r.h.s. of (\ref{2.6}) vanishes.
Dropping the second term on the l.h.s. (\ref{2.6}) implies
\begin{equation}\label{2.7}
    \dt \io {1 \over \t^{s-1}}dx\le -(s-1)\io g \t^{-s}dx.
\end{equation}
Since $g\ge 0$ and $\t\ge \t*>0$ we obtain 
\[
\dt \io {1 \over \t^{s-1}}dx\le 0.
\]
Integrating this with respect to time implies 
\[
\t _*\le |\t(0)|_{\infty,\O}\le |\t(t)|_{\infty,\O}.
\]
Hence, the l.h.s. of (\ref{2.3}) holds and the lemma is proved.
\end{proof}
\begin{lemma}\label{lemma 2.3}
    Assume that $g\in L_2(\Ot),\,g\ge 0,\,\iot gdxdt'<\iy,$\linebreak$ \t(0)\in L_2(\O),$  $\divv v=0,\,v\cdot\bar n|_S=0$.

Then solutions to (\ref{1.2}) , (\ref{1.3})$_3$,(\ref{1.4})$_2$ satisfy the estimate
\begin{equation}\label{2.8}
    \|\t\|_{V(\Ot)}\le\bigg(|g|_{2,\Ot}+\bigg|\intop _0^t gdxdt'\bigg|^2 +|\t(0)|^2_{2,\O}\bigg)\equiv D_0.
\end{equation}
\end{lemma}
\begin{proof}
    Multiplying (\ref{1.2}) by $\t$ and integrating over $\O$ and using the boundary conditions yield
    \begin{equation}\label{2.9}
        {1 \over 2}\dt|\t|^2_{2,\O}+\kappa |\nabla \t|^2_{2,\O}\le \io g\t dx.
    \end{equation}
We write \rf{2.9} in the form 
\begin{equation}\label{2.10}
     {1 \over 2}\dt|\t|^2_{2,\O}+\kappa |\nabla \t|^2_{2,\O}\le \io g\bigg(\t-\oint\t\bigg) dx+\io g\oint \t dx.
\end{equation}
 Integrating \rf{1.2}   over $\Omega$. and using boundary conditions implies 
 \begin{equation}\label{2.11}
     \dt \oint_{\O}\t dx=\oint _{\O}gdx,
 \end{equation}
 where \[
 \oint_{\O}={1 \over |\O|}\io\quad \text{ and }|\O|=\text{meas }\O.
\]
Integrating \rf{2.11} with respect to time yields
\begin{equation}\label{2.12}
    \oint_{\O}\t dx=\intop_0^t \oint_{\O}g dx dt'+\oint_{\O}\t(0)dx.
\end{equation}
Estimating r.h.s. of \rf{2.10} gives
\begin{equation}\label{2.13}
    \begin{split}
        {1 \over 2}&\dt |\t|^2_{2,\O}+\kappa|\nabla \t|^2_{2,\O}\le\ve \bigg|\t-\oint \t\bigg|^2_{6,\O}\\
        &+c(1/\ve)|g|^2_{6/5,\O}+\io g dx \bigg (\intop_0^t \oint_{\O}g dx dt'+\oint_{\O}\t(0)dx\bigg)\\
    \end{split}
\end{equation}
Hence for sufficiently small $\ve$ we have 
\begin{equation}\label{2.14}
    \|\t\|^2_{V(\Ot}\le c\int_0^t |g|^2_{2,\O}dt'+c\bigg |\int_0^t \oint_{\O}
g dx dt'\bigg|^2 +c|\t(0)^2_{2,\O}
\end{equation}
This implies \rf{2.8}  and ends the proof.
\end{proof}
\begin{lemma}\label{lemma 2.4}
    Let the assumptions of Lemma \rf{2.2} hold. Assume that there exists a sufficiently regular solution to problem (\ref{1.1}) , (\ref{1.3})$_{1,2}$,(\ref{1.4}). Let $f\in L_{2,1}(\Ot),\,v(0)\in L_2(\O)$.
Then solutions to  the problem satisfy the estimate
\begin{equation}\label{2.15}
    \begin{split}
        |v(t)|^2_{2,\O}&+\nu \iot \big(|\nabla v_r|^2+|\nabla v_\vp|^2 +|\nabla v_z|^2\big)dx dt'\\
      & +\nu\iot\bigg({v_r^2 \over r^2}+{v_{\vp}^2 \over r^2}\bigg )dx dt'
       \le 3 (\alpha(\t_*,\t^*)|f|^2_{2,1,\Ot}\\
       &+2|v(0)|^2_{2,\O}\equiv D_1^2.
    \end{split}
\end{equation}\end{lemma}
\begin{proof}
Multiplying (1.7) by $v_{r}, v_{\varphi}, v_{z}$, respectively, integrating over $\Omega$ and adding yield

\begin{equation}\lb{2_16}
 \begin{aligned}
\fr 1 2 & \dt \lp{v}{2,\O}{2}+\nu \io \ll(\lp{\nb v_r}{}{2}+\lp{\nb v_\vp}{}{2}+\lp{\nb v_z}{}{2}\rr)dx \\
&+\nu \io \ll (\fr{v_r ^2}{r^2}  + \fr{v_\vp ^2}{r^2}\rr)dx \leqslant \al(\t_*,\t^*)\lp{f}{2,\O}{}\lp{\vphi}{2,\O}{}.
\end{aligned}   
\end{equation}
Then we obtain
$$
\frac{d}{d t}|v|_{2, \Omega} \leqslant \alpha\left(\theta_*, \theta^{*}\right)|f|_{2, \Omega} .
$$
Integrating the inequality with respect to time gives
\begin{equation}\lb{2_17}
    |v|_{2, \Omega} \leqslant \alpha\left(\theta_{*}, \theta^{*}\right)|f|_{2,1, \Ot}+|v(0)|_{2, \Omega}.
\end{equation}
Integrating \eqref{2_16} with respect to time and using \eqref{2_17} yield \eqref{2.15}.
\end{proof}
\begin{lemma}\label{lemma 2.5}
    Let the assumptions of Lemma \ref{2.2} hold. Let \rf{1.21} \linebreak and\rf{2.15}  be satisfied. Let $f_0\in L_{\infty,1}$, $u(0)\in L_{\infty}(\O).$ Using that $v_\vp |_{r=0}
=0$ the following estimate holds
\begin{equation}\label{2.16}
    |u(t)|_{\infty,\O}\le\alpha(\t_*,\t^*)|f_0|_{\infty,1,\Ot}+|u(0)|_{\infty,\O}\equiv D_2.
\end{equation}
\end{lemma}
\begin{proof}
    Multiplying the swirl equation \rf{1.9} by $u|u|^{s-2},\,s>2,$ integrating over $\O$ and by parts, we obtain 
    \[
    \begin{split}
        {1 \over s}\dt|u|^s_{s,\O}+{4\nu(s-1) \over s^2}|\nabla|u|^{s/2}_{2,\O}
&+{\nu \over s}\io \ll (|u|^s\rr )_{,r}dr dz\\
&=\io \alpha f_0u|u|^{s-2}dx.
\end{split}
    \]
    Noting that $u|_{r=0}=u|_{R=0}=0$ (by \rf{1.3} and \rf{1.21}, we see that the last term on the l.h.s. vanishes and using \rf{2.3} we obtain 
    \[
    \dt |u|_{s,\O}\le\alpha (\t_*,\t^*)|f_0|_{s,\O}.
    \]
    Integrating in time and taking $s\to \infty$ gives \rf{2.16}.
\end{proof}
\begin{lemma}\label{lemma 2.6}(Energy estimates for $\psi$ and $\psi_1$)
    For any $v$ satisfying \rf{2.15},
\begin{align}
    \|\psi\|^2_{1,\O}+|\ps|_{2,\O}^2&\le c D_1^2,\label{2.17}\\
    \|\psi_{,z}\|^2_{1,2,\Ot}+|\ps_{,z}|^2_{2,\Ot}&\le c D_1^2\label{2.18}.
    \end{align}
    
\end{lemma}
\begin{proof}
    Multiplying \rf{1.14}$_1$ by $\psi$ and integrating over $\O$ yields
    \[
    \begin{split}
        |\nabla \psi |^2_{2,\O}+|\ps|^2_{2,\O}&=\io \o_\vp \psi dx =\io (v_{r,z}-v_{z,r})\psi dx=\io (v_z\psi_{,r}-v_r\psi_{,z})dx\\ 
       & \le (|\psi_{,r}|^2_{2,\O}+|\psi_{,z}|_{2,\O}^2)/2+c(|v_r|^2+|v_z|^2),
    \end{split}
    \]
    where we integrated by parts and used the boundary condition $\psi|_S=0$  in the third equality. Hence \rf{2.17} holds.

    For \rf{2.18} we differentiate \rf{1.14}$_1$ with respect to $z$, multiply by $\psi_{,z}$ and integrate over $\Ot$ to obtain 
    \[
    \begin{split}
        &\iot|\nabla \psi_{,z} |dx dt'+\iot|\ps_{,z}|^2dx dt'=\io \o_{\vp,z} \psi_{,z} dxdt'\\
        &=-\io \o_{\vp} \psi_{,zz} dxdt'\le |\psi_{,zz}|^2_{2,\O}/2+c|\o_\vp|^2_{2,\Ot}
    \end{split}
    \]
    where we used boundary condition $\o_\vp|_S=0  $ in the second equality. Using \rf{2.15} yields \rf{2.18}. This ends the proof.
\end{proof}
\subsection{Inequalities}\label{ss2.3}
\begin{lemma}\label{lemma 2.7}(Hardy inequality, \cite{BIN} Lemma 2.16)
    Let $p\in[1,\infty],$\linebreak$\beta \ne 1/p$, and let $F(x):=\int_0^xf(y)dy$ for $\beta>1/p$ and $F(x):=\int_x^\infty f(y)dy$
for  $\beta<1/p$. Then
\begin{equation}\label{2.19}
    |x^{-\beta}F|_{p,\R_+}\le {1 \over |\beta-1/p|}|x^{-\beta+1}f|_{p,\R_+}.
\end{equation}
\end{lemma}
\begin{lemma}\label{lemma 2.8} (Sobolev interpolation, see Sect. 15 in\cite{BIN}) 
Let $\t$ satisfy the equality
\begin{equation}\label{2.20}
    {n \over p}-r=(1-\t){n \over p_1} +\t({n \over p_2 }-l),\;{r \over l}\le \t \le 1
\end{equation}
where $1\le p_1\le \infty,\,1\le p_2\le \infty,\,0\le r\le l$.
Then the interpolation holds
\begin{equation}\label{2.21}
    \sum_{|\alpha|=r}|D^\alpha f|_{p,\O}\le c|f|^{1-\t}_{p_1,\O} \|f\|^{\t} _{W^l_{p_2}(\O)},
\end{equation}
    where $\O\subset \R^n,\;D^\alpha  f=\partial^{\alpha_1}_{x_1}\ldots\partial^{\alpha_n}_{x_n}f, \;
|\alpha|=\alpha_1+\ldots \alpha_n.$
\end{lemma}
\begin{lemma}\label{lemma 2.9} (Hardy interpolation,, see Lemma 2.4 in\cite{CFZ})

Let $f\in C^\infty((0,R)\times(-a,a)), f|_{r>R}=0$ . Let  $0\le s\le p,$  $s<2,\,q\in\bigg[p,{p(3-s) \over 3-p}\bigg].$ 
Then there exist positive constant $c=c(q,s)$ such that 
\begin{equation}
    \bigg(\io {|f|^q \over r^s} dx\bigg)^{1/q}\le c|f|_{p,\O}^{{3-s \over q }-{3 \over p}+1}|\nabla f|_{p,\O}^{{3 \over p}-{3-s \over q}}
\end{equation}
where f does not depend on $\vp$.
\end{lemma}
\subsection{Notation of constants} \label{ss2.4}
We will use the following notations for constants depending only on data and forcing:
\[
\begin{split}
&D_0:=|g|_{2,\Ot}+\bigg |\iot
g dx dt'\bigg |+|\t(0)|_{2,\O}\;\quad \text{(see \rf{2.8})},\\
&D_1:=\phi(\t_*,\t^*)\|f\|_{L_1(0,t;L_2(\O))}+\|v(0)\|,\\
&\hspace{1 cm}\text{                                      where } |\alpha(\t)|\le \phi(\t_*,\t^*)\quad\text{(see \rf{2.15})},
\\
&D_2:=\phi(\t_*,\t^*)\|f_0\|_{L_1(0,t;L_\infty(\O))}+\|u(0)\|_{L_\infty(\O)},\\
& \hspace{1cm} f_0=rf_\vp,u=rv_\vp\quad\text{(see \rf{2.16})},\\
&B_1:={1 \over \sqrt{\nu}}\phi(\t_*,\t^*)\|\bar f\|_{L_\infty(0,t;L_3(\O))},\; \bar f=f/r\quad\text{(see Lemma \ref{lemma 4.1})},\\
&D_3:={1 \over \sqrt{\nu}}\phi(\t_*,\t^*)\|\bar F_r\|_{L_2(0,t;L_{6/5}(\O))}+\|\bar F_\vp
\|_{L_2(0,t;L_{6/5}(\O))}\\
&\qquad +D_2|\Gamma(0)_{2,\O}+|\Phi|_{2,\O}\quad\text{(see Lemma \ref{lemma 4.1})},\\
&D_4^2:=D_1^2D_2^2+|u{,z}(0)|_{2,\O}^2+\phi(\t_*,\t^*)|f_0|^2_{2,\Ot}\quad\text{(see \rf{5.5})},\\
&D_5^2:=D_1^2+D_1^2D_2+D_1^2D_2^2\quad\text{(see \rf{5.11})},\\
&D_6^2:={D_2^2 \over min\{1,D_2^2\}}D_2^{1-\ve}{R^{\ve_2}\over \ve_2}\quad\text{(see \rf{4.11})},
\\
&D_7^2:={D_2^2 \over min\{1,D_2^2\}}B_1^2\quad\text{(see \rf{4.11})},
\\
&D_8^2:={D_2^2 \over min\{1,D_2^2\}}D_3^2\quad\text{(see \rf{4.11})},
\\
&D_9^2:={\phi(\t_*,\t^*)\over \nu}|f_\vp|_{3,\infty,\Ot}\quad\text{(see \rf{6.2})},
\\
&D_{10}^2:=(D_4+D_5\|f_\vp\|_{L_2(0,t;L_3(S_1))}+{1 \over \nu }(|F_r|^2_{6/5,2,|ot}+|F_z|^2_{6/5,2,|ot})\\
&\qquad +|\o_r(0)|^2_{2,\O}+|\o_r(0)|^2_{2,\O}\quad\text{(see \rf{6.2})},
\\
&D_{11}=D_2^{1/2}    \phi(\t_*,\t^*)\ll|\fr{f_\vp}r\rr|_{\iy,1,\Ot}^{1/2}+\lp{v_\vp(0)}{\iy,\O}{} \quad\text{(see \rf{6.18})},\\
&D_{12}=\fr{D_2^2 D_1^2}{c_0^{s-2}}+\fr{\lp{f_\vp}{10/7,\Ot}{}D_1}{c_0^{s-2}}+\fr12\lp{\vphi(0)}{s,\O}{}\quad\text{(see \rf{6.21})},
\end{split}
\]
where $\phi$ is an increasing positive function.

\section{Estimates for the modified stream function $\psi_1$}\label{s3}

Here we introduce some estimates of $\psi_1$ in terms of $\Gamma$.

\subsection{Weighted Sobolev estimates for $\psi_1$}
\begin{lemma}[see Lemma 4.2 \cite{NZ}]\label{l2.18}
If $\psi_1$ is a sufficiently regular solution to \eqref{1.17}, then
\begin{equation}
\intop_\Omega(\psi_{1,zzz}^2+\psi_{1,rzz}^2)r^{2\mu}+2\mu(1-\mu)\intop_\Omega\psi_{1,zz}^2r^{2\mu-2}\le c\intop_\Omega\Gamma_{,z}^2r^{2\mu}.
\label{3.1}
\end{equation}
\end{lemma}

\begin{proof}
We differentiate (\ref{1.17}) with respect to $z$, multiply by $-\psi_{1,zzz}r^{2\mu}$ and integrate over $\Omega$ to obtain
\begin{equation}\lb{3.2}
\begin{split}
&\intop_\Omega\psi_{1,rrz}\psi_{1,zzz}r^{2\mu}+\intop_\Omega\psi_{1,zzz}^2r^{2\mu}\\
&\quad+3\intop_\Omega{1\over r}\psi_{1,rz}\psi_{1,zzz}=-\intop_\Omega\Gamma_{,z}\psi_{1,zzz}r^{2\mu}.
\end{split}
\end{equation}
In view of \eqref{1.17}, the first integral on the left-hand side of (\ref{3.2}) equals
\begin{equation}
\begin{split}
-&\intop_\Omega\psi_{1,rrzz}\psi_{1,zz}r^{2\mu}=-\intop_\Omega(\psi_{1,rzz}\psi_{1,zz}r^{2\mu+1})_{,r}\d r\d z\\
&\quad+\intop_\Omega\psi_{1,rzz}^2r^{2\mu}+(2\mu+1)\intop_\Omega\psi_{1,rzz}\psi_{1,zz}r^{2\mu}\d r\d z.
\end{split}
\label{2.29}
\end{equation}
Since $\psi_1|_{r=R}=\psi_{1,r}|_{r=0}=0$ (by \eqref{1.17} and \eqref{1.20}), the first term on the right-hand side of  \eqref{2.29} vanishes. Integrating by parts with respect to $z$ in the last term on the left-hand side of (\ref{3.2}) and using \eqref{1.17}, it takes the form
\begin{equation}
-3\intop_\Omega\psi_{1,rzz}\psi_{1,zz}\d r\d z.
\label{2.30}
\end{equation}
Using (\ref{2.29}) and (\ref{2.30}) in (\ref{3.2}) yields
\begin{multline}\label{2.31}
\intop_\Omega(\psi_{1,zzz}^2+\psi_{1,rzz}^2)r^{2\mu}+2(\mu-1)\intop_\Omega\psi_{1,rzz}\psi_{1,zz}r^{2\mu}\d r\d z\\
=-\intop_\Omega\Gamma_{,z}\psi_{1,zzz}r^{2\mu}.
\end{multline}
The second term on the left-hand side of (\ref{2.31}) equals
\begin{multline}(\mu-1)\intop_\Omega\partial_r(\psi_{1,zz}^2)r^{2\mu}\d r\d z \\
=(\mu-1)\intop_\Omega\partial_r(\psi_{1,zz}^2r^{2\mu})\d r\d z+2\mu(1-\mu)\intop_\Omega\psi_{1,zz}^2r^{2\mu-1}\d r\d z,
\label{2.32}
\end{multline}
where the first integral vanishes because $\psi_{1,zz}|_{r=R}=0$ (recall (\ref{1.22}) )and $\psi_{1,zz}^2r^{2\mu}|_{r=0}=0$ (recall \eqref{1.22}). Using (\ref{2.32}) in (\ref{2.31}) and applying the H\"older and Young inequalities to the r.h.s. of (\ref{2.31}), we obtain (\ref{3.1}), as required.
\end{proof}

\subsection{Elliptic estimates for the modified stream function $\psi_1$}\label{s3.2}

We recall that the modified stream function $\psi_1$ is a solution to the problem \eqref{1.17},
\[
-\Delta\psi_1-{2\over r}\psi_{1,r}=\Gamma,\qquad \left. \psi_1 \right|_S=0.
\]
In this section we prove $H^2$ and $H^3$ elliptic estimates for $\psi_1$, in cylindrical coordinates.

\begin{lemma}[$H^2$ elliptic estimate on $\psi_1$, see Lemma 3.1 in \cite{Z1}]\label{l3.1}
If $\psi_1$ is a sufficiently regular solution to \eqref{1.17} then
\begin{equation}\label{psi1_h2}
\intop_\Omega \left( \psi_{1,rr}^2+\psi_{1,rz}^2+\psi_{1,zz}^2 + \frac{\psi_{1,r}^2}{r^2} \right) +\int_{-a}^a \left( \left.\psi_{1,z}^2\right|_{r=0} +  \left.\psi_{1,r}^2\right|_{r=R} \right) \d z \le c|\Gamma|_{2,\Omega}^2.
\end{equation}
\end{lemma}

\begin{proof}
We multiply \eqref{1.17} by $\psi_{1,zz}$ and integrate over $\Omega$ to obtain
\begin{equation}
-\intop_\Omega \left( \psi_{1,rr}\psi_{1,zz}+ \psi_{1,zz}^2 + 3 \frac{ \psi_{1,r}}r \psi_{1,zz}\right) =\intop_\Omega\Gamma\psi_{1,zz}.
\label{3.3}
\end{equation}
Integrating by parts with respect to $r$ in the first term gives
\[
\begin{split}
-&\intop_\Omega(\psi_{1,r}\psi_{1,zz}r)_{,r}\d r\d z+\int_\Omega\psi_{1,r}\psi_{1,zzr} + \int_\Omega\psi_{1,r}\psi_{1,zz}\d r\d z\\
&-\intop_\Omega  \psi_{1,zz}^2  - 3\intop_\Omega\psi_{1,r}\psi_{1,zz}\d r\d z=\intop_\Omega\Gamma \psi_{1,zz}.
\end{split}\]
Thus
\begin{multline}\label{3.4}
-\int_{-a}^a \left[ \psi_{1,r}\psi_{1,zz}r\right]_{r=0}^{r=R}\d z+\intop_\Omega\psi_{1,r}\psi_{1,zzr}-\int_\Omega\psi_{1,zz}^2 \\-2\intop_\Omega\psi_{1,r}\psi_{1,zz}\d r\d z
= \intop_\Omega\Gamma\psi_{1,zz}.
\end{multline}
We note that the the first integral vanishes since $\psi_{1,r}|_{r=0}=0$ (recall expansion \eqref{1.23}) and $\psi_{1,zz}|_{r=R}=0$. We now integrate by parts with respect to $z$ in the second and the last terms on the left-hand side and use that $\psi_{1,r}|_{S_2}=0$ (since $\psi_1|_{S}=0$, recall \eqref{1.17}), and we multiply by $-1$, to   obtain
\begin{equation}\label{3.7}
\intop_\Omega(\psi_{1,zr}^2+\psi_{1,zz}^2) -2\intop_\Omega\psi_{1,rz}\psi_{1,z}\d r\d z =-\intop_\Omega\Gamma \psi_{1,zz}.
\end{equation}
We note that the last term on the left-hand side equals
\[
-\intop_\Omega (\psi_{1,z}^2)_{,r} \d r \d z = -\int_{-a}^a\left[ \psi_{1,z}^2 \right]_{r=0}^{r=R}\d z=\intop_{-a}^a\left. \psi_{1,z}^2\right|_{r=0}\d z,
\]
since $\psi_{1,z}|_{r=R}=0$. Applying this in \eqref{3.7}, and using the Young inequality to absorb $\psi_{1,zz}$ by the left-hand side, we obtain
\begin{equation}
\intop_\Omega \left( \psi_{1,rz}^2+\psi_{1,zz}^2\right) +\intop_{-a}^a\left. \psi_{1,z}^2\right|_{r=0}\d z\le c|\Gamma|_{2,\Omega}^2.
\label{3.8}
\end{equation}
We now multiply \rf{1.17}$ _1$ by $\psi_{1,r}/r$ and integrate over $\Omega$ to obtain
\begin{equation}\label{3.9} 3\intop_\Omega \frac{\psi_{1,r}^2}{r^2} = -\intop_\Omega\left( \psi_{1,rr}\frac{\psi_{1,r}}r  + \psi_{1,zz}\frac{\psi_{1,r}}r + \Gamma\frac{ \psi_{1,r}}r \right).
\end{equation}
The first term on the right-hand side equals
\[
-{1\over 2}\intop_\Omega\partial_r\psi_{1,r}^2\d r\d z=-{1\over 2}\intop_{-a}^a \left[ \psi_{1,r}^2\right]_{r=0}^{r=R}\d z=-{1\over 2}\intop_{-a}^a\left. \psi_{1,r}^2\right|_{r=R}\d z.
\]
where we used that $\psi_{1,r}|_{r=0}=0$ (recall expansion \eqref{1.23}) in the last equality. As for the other terms on the right-hand side of \eqref{3.9} we apply Young's inequality to absorb $\psi_{1,r}/r$ by the left-hand side. We obtain
\[\intop_\Omega\frac{\psi_{1,r}^2}{r^2} +{1\over 2}\intop_{-a}^a \left. \psi_{1,r}^2\right|_{r=R}\d z\les c\ll(|\psi_{1,zz}|_{2,\Omega}^2+|\Gamma |_{2,\Omega}^2\rr).
\]
The claim \eqref{psi1_h2} follows from this, \eqref{3.8}, and from the equation $\eqref{1.17}_1$ for $\psi_1$, which lets us estimate $\psi_{1,rr}$ in terms of $\psi_{1,zz}$, $\psi_{1,r}/r$.
\end{proof}

\begin{lemma}[$H^3$ elliptic estimates on $\psi_1$]\label{l3.2}
If $\psi_1$ is a sufficiently regular solution to \eqref{1.17} then
\begin{equation}
\intop_\Omega(\psi_{1,zzr}^2+\psi_{1,zzz}^2)+\intop_{-a}^a\psi_{1,zz}^2\bigg|_{r=0}\d z\le c|\Gamma_{,z}|_{2,\Omega}^2
\label{3.13}
\end{equation}
and
\begin{multline}
\intop_\Omega(\psi_{1,rrz}^2+\psi_{1,rzz}^2+\psi_{1,zzz}^2)+ \intop_{-a}^a\psi_{1,zz}^2\bigg|_{r=0}\d z+\intop_{-a}^a \psi_{1,rz}^2\bigg|_{r=R}\d z\cr
\le c|\Gamma_{,z}|_{2,\Omega}^2.\\
\label{3.14}    
\end{multline}
as well as
\begin{equation}
\bigg|{1\over r}\psi_{1,rz}\bigg|_{2,\Omega}\le c|\Gamma_{,z}|_{2,\Omega}
\label{3.24}
\end{equation}
\end{lemma}

\begin{proof}
First we show (\ref{3.13}). We differentiate $\eqref{1.17}_1$ with respect to $z$, multiply by $-\psi_{1,zzz}$ and integrate over $\Omega$ to obtain
\begin{multline}\label{3.15}
\intop_\Omega\psi_{1,rrz}\psi_{1,zzz} +\intop_\Omega\psi_{1,zzz}^2+3\intop_\Omega{1\over r}\psi_{1,rz}\psi_{1,zzz}=-\intop_\Omega\Gamma_{,z}\psi_{1,zzz}.
\end{multline}
Integrating by parts with respect to $z$ in the first term yields
\begin{equation} \label{3.16}
\intop_\Omega\psi_{1,rrz}\psi_{1,zzz}= \intop_\Omega(\psi_{1,rrz}\psi_{1,zz})_{,z}- \intop_\Omega\psi_{1,rrzz}\psi_{1,zz},
\end{equation}
where the first term vanishes due to \eqref{1.17}. Integrating the last integral in  \eqref{3.16} by parts with respect to $r$  gives
$$
-\intop_\Omega(\psi_{1,rzz}\psi_{1,zz}r)_{,r}\d r\d z+\intop_\Omega\psi_{1,rzz}^2+ \intop_\Omega\psi_{1,rzz}\psi_{1,zz}\d r\d z,
$$
where the first integral vanishes, since $\psi_{1,rzz}|_{r=0}=\psi_{1,zz}|_{r=R}=0$ (recall \eqref{1.23} and \eqref{1.17}).
Thus, (\ref{3.15}) becomes
\begin{multline}\label{3.17}
\intop_\Omega(\psi_{1,rzz}^2+\psi_{1,zzz}^2)+\intop_\Omega\left( \psi_{1,rzz}\psi_{1,zz} + 3\psi_{1,rz}\psi_{1,zzz}\right) \d r\d z\\=-\intop_\Omega\Gamma_{,z}\psi_{1,zzz}.
\end{multline}
Integrating by parts with respect to $z$ in the last term on the left-hand side of \eqref{3.17} and using that $\psi_{1,zz}|_{S_2}=0$ (recall \eqref{1.17}) we get
\begin{equation}\label{3.18}
\intop_\Omega(\psi_{1,rzz}^2+\psi_{1,zzz}^2)-\intop_\Omega\partial_r\psi_{1,zz}^2\d r\d z= -\intop_\Omega\Gamma_{,z}\psi_{1,zzz}.
\end{equation}
Recalling \eqref{1.17} that  $\psi_{1,zz}|_{r=R}=0$, and using Young's inequality to absorb $\psi_{1,zzz}$ we obtain
\[
\intop_\Omega \left( \psi_{1,rzz}^2+\psi_{1,zzz}^2\right) +\intop_{-a}^a\left. \psi_{1,zz}^2\right|_{r=0}\d z\les c |\Gamma_{,z}|_{2,\Omega}^2.
\]
which gives \eqref{3.13}.

As for \eqref{3.14}, we differentiate \rf{1.17}$_1$ with respect to $z$, multiply by $\psi_{1,rrz}$ and integrate over $\Omega$ to obtain
\begin{equation}\label{3.19}
-\intop_\Omega \left( \psi_{1,rrz}^2+ \psi_{1,zzz}\psi_{1,rrz}+3{1\over r}\psi_{1,rz}\psi_{1,rrz}\right) =\intop_\Omega\Gamma_{,z}\psi_{1,rrz}.
\end{equation}
We integrate the second term on the left-hand side by parts in $z$, and recall \eqref{1.17} that  $\psi_{1,zz}|_{S_2}=0$, to get

\begin{multline*}
-\intop_\Omega\psi_{1,zzz}\psi_{1,rrz}=\intop_\Omega\psi_{1,zz}\psi_{1,rrzz}\\
=\intop_\Omega(\psi_{1,zz}\psi_{1,rzz}r)_{,r}\d r\d z-\intop_\Omega\psi_{1,rzz}^2- \intop_\Omega\psi_{1,zz}\psi_{1,rzz}\d r\d z.
\end{multline*}
We note that the first term on the right-hand side vanishes since $\psi_{1,rzz}|_{r=0}=0$ (recall \eqref{1.23}) and $\psi_{1,zz}|_{r=R}=0$ (recall \eqref{1.17}), and so \eqref{3.19} becomes
\begin{multline}\label{3.21}
\intop_\Omega \left( \psi_{1,rrz}^2+\psi_{1,rzz}^2 \right)+\intop_\Omega \left( \psi_{1,zz}\psi_{1,rzz}+3 \psi_{1,rz}\psi_{1,rrz} \right) \d r\d z \\=
-\intop_\Omega\Gamma_{,z}\psi_{1,rrz}.
\end{multline}
Since the second term above equals
\[
{1\over 2}\intop_{-a}^a\left[ \psi_{1,zz}^2\right]_{r=0}^{r=R}\d z=-{1\over 2}\intop_{-a}^a\left. \psi_{1,zz}^2\right|_{r=0}\d z
\]
(as $\psi_{1,zz}|_{r=R}=0$, recall \eqref{1.17}), and the last term on the left-hand side of \eqref{3.21} equals
\[
\frac32 \intop_\Omega\partial_r\psi_{1,rz}^2\d r\d z=\frac32 \intop_{-a}^a\left[ \psi_{1,rz}^2\right]_{r=0}^{r=R}\d z=\frac32 \intop_{-a}^a\left. \psi_{1,rz}^2\right|_{r=R}\d z
\]
(as $\psi_{1,rz}|_{r=0}=0$, recall \eqref{1.23}), \eqref{3.21} becomes
\begin{multline}\label{3.22}
\intop_\Omega(\psi_{1,rrz}^2+\psi_{1,rzz}^2)+ \intop_{-a}^a \left( -\frac12 \left. \psi_{1,zz}^2\right|_{r=0} + \frac32 \left. \psi_{1,rz}^2\right|_{r=R} \right) \d z\\=-\intop_\Omega\Gamma_{,z}\psi_{1,rrz}.
\end{multline}
We now use Young's inequality to absorb $\psi_{1,rrz}$ by the left-hand side to obtain \eqref{3.14}, which in turn implies \eqref{3.24} by differentiating $\eqref{1.17}_1$ in $z$.
\end{proof}

\section{Energy Estimates for $\Phi$ and $\G$}\label{s4}
\begin{lemma}\label{lemma 4.1}(Energy Estimate for $\Phi,\G$).
Let $\t_*,\t^*,\t_*\le\t^*$be given positive numbers defined in Lemma \ref{lemma 2.2}.
Let $\t$ be a solution to \rf{1.2},\rf{1.3}$_3$,\rf{1.4}$_2$ such that $\t_*\le \t\le\t^*$ and $\nabla \t\in L_2(\Ot)$.
If $v$ is regular solution to     \rf{1.1},\rf{1.3}$_{1,2}$,\rf{1.4}$_1$ such that $v_\vp\in L_\infty(\Ot)$   for $t\in (0,T)$,  then,  for every $t\in (0,T)$ ,
\begin{equation}\label{4.1}
    \begin{split}
        D_2^2\|\G\|_{V(\Ot)}+\|\Phi\|_{V(\Ot)} &\le cD_2^2\left(1+{|v_\vp|_{\infty,\Ot}^{2\ve_0}R^{2\ve_0} \over \ve_0^2D_2^{2\ve_0}}\right)\cdot \\&\cdot \left (I+B_1^2|\nabla\t|+D_3^2\right),
    \end{split}
\end{equation} 
where 
\[
I=\left |\iot {v_\vp \over r}\Phi\G dx dt'\right|.
\]
Moreover, if $v_\vp\in L_\infty(0,t;L_d(\O))$ for some $d>3$ and $\ve_1,\ve_2>0$ are sufficiently small such that
\[
\begin{split}
   & \theta_0\:= \left(1-{3\over d}\right)\varepsilon_1-{3\over d}\varepsilon_2>0, \quad 1+\frac{\varepsilon_2}{\varepsilon_1} < \frac{d}3,\\
&\t_0<1 \quad\text{ implies } \quad \ve_1\left(1-{3 \over d}\right)<1+{3 \over d}\ve_2,
\end{split}
\]
then
\begin{equation}\label{4.2}
    I \le cD_2^{1-\varepsilon}|v_\varphi|_{d,\infty,\Omega^t}^\varepsilon{R^{\varepsilon_2}\over\varepsilon_2} |\Phi|_{2,\Omega^t}^{\theta_0}\| \Phi \|_{V(\Omega^t )}^{1-\theta_0}\| \Gamma \|_{V(\Omega^t )},
\end{equation}
where $\varepsilon:=\varepsilon_1+\varepsilon_2$.  
Finally, we have 
\[
\begin{split}
   & B_1^2={1 \over \nu}
\phi^2(\t_*,\t^*)|\bar f|_{3,\infty,\Ot}\\
&D_3^2=\phi^2(\t_*,\t^*)\left(|\bar F_r|_{6/5,2,\Ot}^2+|\bar F_\vp|_{6/5,2,\Ot}^2\right .\\
&\left .\qquad+ D_2^2|\G(0)|_{2,\O}^2+|\Phi(0)|_{2,\O}^2\right),
\end{split}
\]
where $\bar f=f/r,\;\bar F=F/r.$
\end{lemma}
\begin{proof}
    We multiply \eqref{1.11} by $\Phi$ and integrate over $\Omega$ to obtain
\begin{equation}\label{4.3}
    \begin{split}
        \fr{1}{2}\dt&\lp{\Phi}{2,\O}{2}+\nu\lp{\nb \Phi}{2,\O}{2}-\nu \intop_{-a}^a\Phi^2\bigg|_{r=0}^{r=R}dz\\
   & =\io(\o_r  \pa_r +\o_z\pa_z)\fr{v_r}{r}  \Phi dx\\
   &+\io\dot \al\t_{,z}\bar f_{\vp}\Phi dx +\io\al \bar F_r\Phi dx
    \end{split}
\end{equation}   
where the last term on the l.h.s. equals $\intop_{-a}^a\Phi^2|_{r=0}d z$, due to \eqref{1.6}, \rf{1.13}.  
Recalling \rf{1.5} we can integrate in the first term on the r.h.s.  of \rf{4.3} by parts 
\[
\begin{split}
\intop_\Omega&(\omega_r\partial_r+\omega_z\partial_z){v_r\over r}\Phi =\intop_\Omega\left[] -v_{\varphi,z}\left( \frac{v_r}{r} \right)_{,r}+\frac{(rv_\varphi)_{,r}}{r^2}v_{r,z} \right] \Phi  r dr dz\\
&=\intop_\Omega v_\varphi\left(\left(\frac{v_r}{r}\right)_{,rz}\Phi+\left( \frac{v_r}{r} \right)_{,r} \Phi_{,z} \right)-\intop_\Omega v_\varphi\bigg(\left(\frac{v_r}{r}\right)_{,rz}\Phi+\left( \frac{v_r}{r} \right)_{,z} \Phi_{,r}\bigg)\\
&=-\intop_\Omega v_\varphi\left[ \psi_{1,zr} \Phi_{,z}-\psi_{1,zz} \Phi_{,r}\right]\equiv I_1,
\end{split}
\]
where in the second equality we used that $\Phi|_{S_2}=0$ (recall \rf{1.13}) and
\[
\intop_{-a}^a\left[ rv_\varphi\partial_z{v_r\over r}\Phi\right]_{r=0}^{r=R}dz=0
\]
because $v_\vp|_{r=R}=0$  by \rf{1.3} and $rv_\varphi\left({v_r\over r}\right)_{,z}\Phi|_{r=0}=0$ by \rf{1.20},\rf{1.21}. \linebreak Using \rf{3.14},\rf{3.24} we get 
\[
\begin{split}
&I_1=-\intop_\Omega v_\varphi\left( \psi_{1,zr} \Phi_{,z}-\psi_{1,zz} \Phi_{,r}\right)dx\\
&\leq \intop_{\Omega} \left| rv_{\varphi } \frac{\psi_{1,rz} }r \Phi_{,z} \right|dx  + \intop_{\Omega} \left| r^{1-\ve_0 } v_{\varphi } \frac{\psi_{1,zz} }{r^{1-\ve_0 }} \Phi_{,r} \right|dx  \\
&\leq |rv_\varphi|_{\infty,\Omega}\bigg|{\psi_{1,rz}\over r}\bigg|_{2,\Omega}|\Phi_{,z}|_{2,\Omega}+|r^{1-\ve_0 } v_\varphi|_{\infty,\Omega}\left|\frac{\psi_{1,zz}}{r^{1-\ve_0}}\right|_{2,\Omega}|\Phi_{,r}|_{2,\Omega} \\
&\le c D_2|\nabla \Phi |_{2,\Omega}\left( |\Gamma_{,z}|_{2,\Omega} +\frac{|v_\varphi|_{\infty,\Omega}^{\ve_0 }}{\ve_0  D_2^{\ve_0 }} |\psi_{1,zzr}r^{\ve_0}|_{2,\Omega} \right) \\
&\le c D_2|\nabla \Phi |_{2,\Omega}|\nabla \Gamma |_{2,\Omega} \left(1+\frac{|v_\varphi|_{\infty,\Omega}^{\ve_0 }R^{\ve_0 }}{\ve_0   D_2^{\ve_0}} \right).
\end{split}
\]
In the second inequality we used the maximum principle  \rf{2.16} and the Hardy inequality  \rf{2.19}, and \eqref{3.13},\eqref{3.24} in the third and fourth inequalities.
Employing Lemma \ref{lemma 2.2} and assuming the  existence of a function $\phi$ such that 
\[
|\dot \al(\t)|+ |\al(\t)|\le \phi(\t_*,\t^*)
\]
we obtain
\[
\left|\io \dot \al \t_{,z} \bar f_\vp \Phi dx \right|\le \delta |\Phi|_{6,\O}^2 +c(1/\delta)\phi^2(\t_*,\t^*)|\t_{,z}|_{2,\O}^2|\bar f_\vp|_{3,\O}^2.
\]
and
\[
\left|\io \al \bar F_r \Phi dx \right|\le \delta \Phi|_{6,\O}^2 +c(1/\delta)\phi^2(\t_*,\t^*)|\bar F_r|_{6/5,\O}
\]
Using the above estimates in \rf{4.3} gives
\begin{equation}\label{4.4} 
\begin{split}
    \dt&|\Phi|_{2,\O}^2+\nu|\nb\Phi|_{2,\O}^2\le {c \over \nu}D_2^2 \lp{\nb \G}{2,\O}{2}\ll(1+{\lp{v_\vp}{\infty,\O}{2\ve_0}R^{2\ve_0}\over \ve_0D_2^{2\ve_0}}\rr)\\
   &+ {c\over \nu}\phi^2(\t_*,\t^*)\lp{\t_{,z}}{2,\O}{2}\lp{\bar f_\vp}{3,\O}{2}+   {c\over \nu}\phi^2(\t_*,\t^*)\lp{\bar F_r}{6/5,\O}{2}.
\end{split}
\end{equation}
Multiplying \rf{1.12} by $\G$, integrating over $\O$ and using boundary conditions, we get
\[\begin{split}
\fr{1}{2}\dt\lp{\G}{2,\O}{2}+\nu \lp{\nb \G}{2,\O}{2}-\nu\int_{-a}^a\G^2\bigg|_{r=0}^{r=R}dz=-2\io \fr{v_r}{r}\Phi
\G dx\\
+\io \dot \al \t_{,z}\bar f_r\G dx-\io \dot \al \t_{,r}\bar f_z\G dx+\io\al\bar F_{\vp}\G dx.
\end{split}
\]
Using Lemma \ref{lemma 2.2} and applying the H\"older and Young inequalities to the last three terms from the r.h.s. and recalling that $\o_\vp |_{r=R}=0$ we derive
\[\begin{split}
&\fr{1}{2}\dt\lp{\G}{2,\O}{2}+\nu \lp{\nb \G}{2,\O}{2}-\nu\int_{-a}^a\G^2\bigg|_{r=0}^{r=R}dz\\
&\le c \ll |\io \fr{v_\vp}{r}\Phi\G dx\rr|+\fr{c}{\nu}\phi^2(\t_*,\t^*)\lp{\nb \t}{2,\O}{2}\lp{\bar f}{3,\O}{2}\\
&+\fr{c}{\nu}\phi^2(\t_*,\t^*)\lp{\bar F_\vp}{6/5,\O}{2}.
\end{split}
\]
Dropping the last term on the l.h.s. , and then multiplying the resulting equation by $cD_2^2\ll(1+{\lp{v_\vp}{\infty,\O}{2\ve_0}R^{2\ve_0}/ \ve_0D_2^{2\ve_0}}\rr)$ and adding to  \rf{4.4} gives 
\[
\begin{split}
    &D_2^2 \dt\lp{\G}{2,\O}{2}+\nu D_2^2 \lp{\nb\G}{2,\O}{2}+\dt\lp{\Phi}{2,\O}{2}+\nu  \lp{\nb\Phi}{2,\O}{2}\\
  & \le cD_2^2\ll(1+{\lp{v_\vp}{\infty,\O}{2\ve_0}R^{2\ve_0}\over \ve_0D_2^{2\ve_0}}\rr)\cdot\bigg(\bigg|\io \fr{\vphi}{r}\Phi\G dx \bigg|\\
  &+\fr{1}{\nu}\phi^2(\t_*,\t^*)\lp{\nb\t}{2,\O}{2} \lp{\bar f}{3,\O}{2}+ \fr{1}{\nu} \phi^2(\t_*,\t^*)\ll(\lp{\bar F_r}{6/5,\O}{2}+\lp{\bar F_\vp}{6/5,\O}{2}\rr)\bigg  ).
\end{split}
\]
Integrating in time gives \rf{4.1}.
We note that
\[
\begin{split}
    I&\le\intop_{\Omega^t}|rv_\varphi|^{1-\varepsilon}|v_\varphi|^\varepsilon \bigg|{\Phi\over r^{1-\varepsilon_1}}\bigg|\bigg|{\Gamma\over r^{1-\varepsilon_2}}\bigg|dxdt'\\
&\le D_2^{1-\varepsilon}\bigg(\intop_{\Omega^t}|v_\varphi|^{2\varepsilon} \bigg|{\Phi\over r^{1-\varepsilon_1}}\bigg|^2dxdt'\bigg)^{1/2}\bigg|{\Gamma\over r^{1-\varepsilon_2}}\bigg|_{2,\Omega^t}\equiv I_1\\
\end{split}
\]
where $\ve=\ve_1+\ve_2$ and $\ve_i,i=1,2$ is a positive number.
Using \rf{2.16} and applying the H\"older inequality in $I_1$ yields
\begin{equation*}
    I_1\le D_2^{1-\varepsilon}\bigg(\intop_{\Omega^t}|v_\varphi|^{2\varepsilon} \bigg|{\Phi\over r^{1-\varepsilon_1}}\bigg|^2dxdt'\bigg)^{1/2}\bigg|{\Gamma\over r^{1-\varepsilon_2}}\bigg|_{2,\Omega^t}
\equiv  I_2.
\end{equation*}
By the Hardy inequality, we obtain
\[
\lp{\G/r^{1-\ve_2}}{2,\Ot}{}\le\fr{c}{\ve_2}\lp{\nb\G r^{{\ve_2}}}{2,\Ot}{}\le \fr{cR^{\ve_2}}{\ve_2}\lp{\nb\G}{2,\Ot}{}.\]
Applying the H\"older inequality yields
\[
\begin{split}
    &\bigg(\intop_0^t\io \lp{v_\vp}{}{2\ve}\ll|\fr{\Phi}{r^{1-\ve_1}}\rr |^2dx dt'\bigg)^{1/2}\\
    &\le \bigg[\intop_0^t   \lp{v_\vp}{2\ve \sigma , \O}{2\ve}\bigg(\io\ll|\fr{\Phi}{r^{1-\ve_1}}\rr |^qdx\bigg)^{2/q}dt'\bigg]^{1/2} \equiv L,
\end{split}
\]
where $1/\sigma+1/\sigma'=1,\;q=2\sg'$. Let $d=2\ve\sg$. Then 
\[
\sg'=\fr{d}{d-2\ve}\;\text{ so }\;q=\fr{2d}{d-2\ve}.
\]
Continuing 
\[
L\le \sup_t \lp{v_\vp}{d,\O}{\ve}\bigg(\intop_0^t \ll|\fr{\Phi}{r^{1-\ve_1}}\rr |_{q,\O}^2dt'\bigg)^{1/2}\equiv L_1L_2.
\]
We now estimate the second factor $L_2$.  For this purpose we use Lemma \ref{2.9} for $p=2,\fr{s}{q}=1-\ve_1 $. Then $q\in[2,2(3-s)]  $   so
\begin{equation}\label{4.5}
    2\le q  \le\fr{6}{3-2\ve_1}
\end{equation}
where $\ve_1\in (0,1).$ Then Lemma \ref{lemma 2.9} implies
\[
\begin{split}
    L_2\le \bigg(\intop_0^t \ll|\fr{\Phi}{r^{1-\ve_1}}\rr |_{q,\O}^2dt'\bigg)^{1/2}\le
c\bigg(\intop_0^t\lp{\Phi}{2,\O}{2(\fr{3-s}{q}-\fr{1}{2})}\lp{\nb\Phi}{2,\O}{2(\fr{3}{2}-\fr{3-s}{q})}\bigg)^{1/2}\\
\le c \lp{\Phi}{2,\O}{\fr{3-s}{q}-\fr{1}{2}}\lp{\nb\Phi}{2,\O}{\fr{3}{2}-\fr{3-s}{q}}\equiv L_2^1
\end{split}
\]
where we used that for $\t_0=\fr{3-s}{q}-\fr{1}{2},\;1-\t_0=\fr{3}{2}-\fr{3-s}{q} $ the H\"older inequality can be applied. 
Since $q=\fr{2d}{d-2\ve}>2.$ we have 
\[
\t_0=\ve_1(1-\fr{3}{d})-\fr{3}{d}\ve_2
\]
Since $\t_0>0$  we have that $d>3$ and  
\begin{equation}\label{4.6}
\ve_1>\fr{3}{d-3}\ve_2    .
\end{equation}
Next, $\t_0\le 1$  implies
\begin{equation}
    \ve_1(1-\fr{3}{d})<1+\fr{3}{d}\ve_2
\end{equation}
which always holds. From the form of $I_2$ and estimate of $L$ we obtain \rf{4.2}. This ends the proof.
\end{proof}
\begin{remark}\label{rm 4.2}
    Introduce the notation
    \begin{equation}
        X(t)=\hp{\Phi}{V(\Ot)}{}+\hp{\G}{V(\Ot)}{}.
        \end{equation}
        Using \rf{4.2} in \rf{4.1} yields 
        \begin{equation}\label{4.9}
            \begin{split}
                \min\{1,D_2^2\}X^2\le cD_2^2\ll(1+{\lp{v_\vp}{\infty,\Ot}{\ve_0}R^{2\ve_0}\over \ve_0^2D_2^{2\ve_0}}\rr)\cdot \\
                \cdot \bigg[D_2^{1-\ve}\lp{v_\vp}{d,\infty,\Ot}{\ve}\fr{R^{\ve_2}}{\ve_2}\lp{\Phi}{2,\Ot}{\t_0}\hp{\Phi}{V(\Ot)}{1-\t_0}\hp{\G}{V(\Ot)}{}\\
                +B_1^2\lp{\nb \t}{2,\Ot}{2}+D_3^2\bigg].
            \end{split}
        \end{equation}
        Continuing we have 
        \begin{equation}\label{4.10}
            \begin{split}
             &   X^2(t)\le c \fr{D_2^2}{ \min\{1,D_2^2\}}\ll(1+{\lp{v_\vp}{\infty,\Ot}{2\ve_0}R^{2\ve_0}\over \ve_0^2D_2^{2\ve_0}}\rr)\cdot \\
              & \cdot \bigg[D_2^{1-\ve}\lp{v_\vp}{d,\infty,\Ot}{\ve}\fr{R^{\ve_2}}{\ve_2}\lp{\Phi}{2,\Ot}{\t_0}X^{2-\t_0}(t)
                 +B_1^2\lp{\nb \t}{2,\Ot}{2}+D_3^2\bigg]\\
                & \le \bigg(1+\lp{v_\vp}{\infty,\Ot}{2\ve_0}\bigg)\bigg[D_6^2\lp{v_\vp}{d,\infty,\Ot}{\ve}\lp{\Phi}{2,\Ot}{\t_0}X^{2-\t_0}(t)+D_7^2\lp{\nb \t}{2,\Ot}{2}+D_8^2\bigg],
            \end{split}
        \end{equation}
        where
        \begin{align}
            &D_6^2=\fr{D_2^2}{ \min\{1,D_2^2\}}D_2^{1-\ve}\fr{R^{\ve_2}}{\ve_2},\label{4.11}\\
            &D_7^2=\fr{D_2^2}{ \min\{1,D_2^2\}}B_1^2\nonumber\\
            &D_8^2=\fr{D_2^2}{ \min\{1,D_2^2\}}D_3^2\nonumber.
        \end{align}
        
\end{remark}
\begin{remark}\label{rm4.13}
    Using \rf{6.1},\rf{6.17} and \rf{6.21} in \rf{4.10} yields 
    \[
    X^2\le\phi_1X^{\fr{3}{2}\ve_0}X^{\t_0/2}X^{2-\t_0}+\phi_2,
    \]
    where $\phi_1,\phi_2$ depend on $D_0,\ldots D_{12}$.
    Since
    \[
    \fr{3}{2}\ve_0+\fr{\t_0}{2}+2-\t_0<2
    \]
    because $\ve_0$ is arbitrary small we obtain 
    \begin{equation}
        X^2(t)\le \phi(D_0,\ldots,D_{12}).
    \end{equation}
\end{remark}
\section{Estimates for swirl $u=rv_\vp$}\label{s5}
We derive energy estimate for $\nb u$. Recall that $\swirl \,   u= rv_\vp$  satisfies  
\begin{equation}\label{5.1}
\begin{split}
       &u_{,t}+v\cdot \nabla u- \nu \Delta u +{2\nu\over r}u_{,r}=\alpha (\t)f_0,\\
       &u=0\qquad \;\;\text{ on }S_1,\\
        &u_{,z}=0\qquad \text{ on }S_2.\\
\end{split}
 \end{equation}
\begin{lemma}[see Lemma 5.1 in \cite{OZ}]\label{lemma 5.1}
Any regular solution $u$ to \rf{5.1} satisfies 
\begin{equation}
\hspace{2cm}|u_{,z}(t)|_{2,\Omega}^2+\nu|\nabla u_{,z}|_{2,\Omega^t}^2\le c D_4^2,
\label{5.2}
\end{equation}
\begin{equation}\label{5.3}
|u_{,r}(t)|_{2,\Omega}^2+\nu\left( |u_{,rr}|_{2,\Omega^t}^2+|u_{,rz}|_{2,\Omega^t}^2\right) \le c D_5^2.
\end{equation}
where  $D_4^2=\fr1\nu(D_1^2+D_2^2+\lp{ u_{,z}(0)}{2,\O} {2} +\phi(\t_*,\t^*)\lp{f_0}{2,\O}{2} ) $
and $D_5$ is defined in \rf{5.11}.
\end{lemma}
\begin{proof}

Differentiating \rf{5.1}  with respect to $z$, multiplying  by $u_{,z}$ and integrating over $\Omega$ to obtain 
\begin{equation}\label{5.4}
\begin{split}
\fr12&\dt\lp {u_{,z}}{2,\O}2-\nu\io\divv(\nb u_{,z}u_{,z})dx+\nu\io |\nb u_{,z}|^2dx\\
&+2\nu\io u_{,zr}u_{,z}drdz+\io v_{,z}\nb u u_{,z}dx
+\fr12 \io v\cdot \nb(u_{,z}^2)dx\\
&=\io(\al(\t)f_0)_{,z}u_{,z}dx
\end{split}
\end{equation} 
The second term vanishes due to the boundary condition $\left. u_{,z}\right|_S=0$ \linebreak (recall \rf{1.3}$_{1,2}$). The fourth term  equals
\[
\nu\intop_\Omega\partial_r(u_{,z}^2)d r d z=\nu\intop_{-a}^au_{,z}^2\bigg|_{r=0}^{r=R}d z=0,
\]
since $u_{,z}|_{r=R}=0$ (see \rf{1.3}$_{1,2}$)   and the fact that $u_{,z}|_{r=0}=0$ (recall \rf{1.21}). Similarly, the sixth term  equals
\[
\frac12 \io v\cdot\nabla u_{,z}^2d x=\frac12 \intop_S v\cdot \bar n u_{,z}^2d S =0,
\]
because  $v\cdot\bar        n|_S=0$  (recall \rf{1.3}$_{1,2}$)  .  Integrating by parts in the fifth term in \rf{5.4}, and noting that the boundary term vanishes (since $u_{,z}=0$ on $S$), we obtain  
\[
\bigg| \intop_\Omega (v_{,z}\cdot\nabla) u\cdot u_{,z} dx\bigg| = \bigg| \intop_\Omega v_{,z}\cdot\nabla u_{,z}udx \bigg| \leq \delta \intop_\Omega|\nabla u_{,z}|^2dx+\frac{c}{\delta}  |u|_{\infty,\Omega}^2\intop_\Omega v_{,z}^2dx.
\]
Finally, integrating the right-hand side of \rf{5.4} by parts in $z$ we obtain 
\[
 \intop_\Omega \ll(\al(\t)f_0\rr)_{,z}u_{,z}dx  =-\intop_\Omega \al(\t)f_0u_{,zz}dx  \le\delta |u_{,zz}|^2_{2,\Omega }+ \frac{c}{\delta }\phi(\t_*,\t^*)|f_0 |_{2,\Omega }^2,
\]
where we used that $\intop_{S_2} \left[ f_0 u_{,z} \right]_{z=-a}^{z=a} rd r=0$ because $u_{,z}|_{S_2}=0$ \linebreak (recall  \rf{1.3}$_2$ ). 
Using the above results in \rf{5.4} gives 
\[
\dt\lp{u_{,z}}{2,\O}{2}+\nu|\nb u_{,z}|_{2,\O}^2\le \fr c\nu \ll(\lp
{u}{\infty,\O}2\lp{v_{,z}}{2,\O}{2}+\phi(\t_*,\t^*)|f_0 |_{2,\Omega }^2\rr)
\]
Integrating in $t\in (0,T)$ gives 
\begin{equation}\label{5.5}
    \begin{split}
        \lp{u_{,z}(t)}{2,\O}2&+\nu \lp{\nb u_{,z}}{2,\Ot}2\le \fr c\nu (\lp{u}{\infty,\Ot}2+\lp{v_{,z}}{2,\Ot}2\\
        &+\lp{u_{,z}(0)}{2,\O}2+\phi(\t_*,\t^*)|f_0 |_{2,\Omega }^2)\le c D_4^2
    \end{split}
\end{equation}
which proves \rf{5.2} using energy estimate \rf{2.15} and maximum principle  \rf{2.16} for the \swirl  $\;u.$
To prove \rf{5.3} we differentiate \rf{5.1}$_1$ with respect to $r$ multiply the resulting equation by  $u_{,r}$ and integrate over $\O$ to obtain
\begin{equation}\label{5.6}
    \begin{split}
        \fr12&\dt\lp{u_{,r}}{2,\O}2+\io v_{,r}\cdot\nb u u_{,r}dx+\io v\cd \nb u_{,r}u_{,r}dx\\
        &-\nu\io (\Delta u)_r u_{,r}dx +2\nu \io u_{,rr}u_{,r}drdz-2\nu\io\fr{u_{,r}^2}{r^2}dx\\&=\io (\al f_0)_{,r}u_{,r}dx.
    \end{split}
\end{equation}
We now examine  particular terms in \rf{5.6}. The second term equals
\[\begin{split}
    &\io v_{,r}\cdot\nb u u_{,r}rdrdz =\io (rv_{r,r}u_{,r}+rv_{z,r}u_{,z})u_{,r}drdz\\
    &=\io\ll[(rv_{r,r}u_{,r})_{,r}+(rv_{z,r}u_{,r})_{,z}\rr]udrdz\equiv I,
    \end{split}
\]
where we integrated by parts with respect in $r$ and $z$, respectively, and used the boundary conditions $u|_{S_1}=u|_{r=0}=0$ (recall \rf{1.3} and \rf{1.21}) and $v_{z,r}|_{S_2}=0$ (recall \rf{1.3}). Continuing, we have 
\[\begin{split}
    I&=-\io \ll[(rv_{r,r})_{,r}+(rv_{z,r})_{,z}\rr]u_{,r}udrdz\\
&-\io \ll[rv_{r,r}u_{,rr}+rv_{z,r}u_{,rz}\rr]udrdz=I_1+I_2
\end{split}
\]
Differentiating the divergence-free equation \rf{1.7}$_4$ in $r$ gives $v_{r,rr} +v_{z,zr}+\fr{v_{r,r}}r -\fr {v_r}{r^2}=0 $. Hence $I_1$ equals 
\[
I_1=-\io\fr{v_r}ru_{,r}udrdz
\]

Using the Young inequality in $I_2$ yields
\[
|I_2|\le\fr\nu 2 (\lp{u_{,rr}}{2,\O}2+\lp{u_{,rz}}{2,\O}2)+\fr c \nu \lp{u}{\infty,\O}2(\lp{v_{r,r}}{2,\O}2+\lp{v_{z,r}}{2,\O}2).
\]
The third term on the l.h.s.  of \rf{5.6} equals 
\[
\fr12 \io v\cd \nb u_{,r}^2dx=\fr12 \io \divv (vu_{,r}^2)dx=0
\]
since $v\cd\bar n|_S=0$. As for the fourth term in \rf{5.6} we have
\[\begin{split}
    -&\io (\Delta u)_{,r}u_{,r}dx=-\io \ll(u_{,rrr}+(\fr1r u_{,r})_{,r}+u_{,rzz}\rr)u_{,r}rdrdz\\
    &=-\io \ll[(u_{,rr}+\fr1r u_{,r})u_{,rr}\rr]_{,r}drdz+\io u_{,rr}(u_{,r}r)_{,r}drdz\\
    &+\io \fr1ru_{,r}(u_{,r}r)_{,r}drdz
+\io u_{,rz}^2dx \\
&=-\intop_{-a}^a\ll[(u_{,rr}+\fr1r u_{,r}){u_{,r}r}\rr]\bigg|_{r=0}^{r=R}dz+\io (u_{,rr}^2+u_{,rz}^2)dx\\
&+\io \fr{u_{,r}^2}{r^2}dx+2\io u_{,rr}u_{,r}drdz.
\end{split}
\]
Using the above expressions in \rf{5.6} yields
\begin{equation}\label{5.7}
    \begin{split}
        \fr12&\dt\lp{u_{,r}}{2,\O}2+\fr\nu 2\io (u_{,rr}^2+u_{,rz}^2)dx-\nu\io \fr{u_{,r}^2}{r^2}dx\\
        &-\nu\intop_{-a}^a\ll[(u_{,rr}+\fr1r u_{,r})u_{,r}r\rr]\bigg|_{r=0}^{r=R}dz+4\nu\io u_{,rr}u_{,r}drdz\\
        &\le \io (\al f_0)_{,r}u_{,r}dx+\io \fr{v_r}r\fr{u_{,r}}rudx
+\fr c\nu\lp{u}{\infty,\O}2(\lp{v_{r,r}}{2,\O}2+\lp{v_{z,r}}{2,\O}2).
\end{split}
\end{equation}
The last term on the l.h.s. of \rf{5.7} equals 
\[
2\nu \intop_{-a}^a
u_{,r}^2\bigg|_{r=0}^{r=R}dz=2\nu\intop_{-a}^au_{,r}^2\bigg|_{r=R}dz
\]
Since the expansion \rf{1.21} implies that 
\begin{equation}\lb{5.8}
    u=b_1(z,t)r^2+b_2(z,t)r^3+\ldots ,
\end{equation}
so that $u_{ ,r}\bigg|_{r=0}=0.$ Moreover, the above expansion used in the fourth term on the l.h.s. of \rf{5.7} implies \begin{equation}\label{5.9}
    \begin{split}
        -&\nu\intop_{-a}^a(u_{,rr}+\fr1ru_{,r})u_{,r}r\bigg|_{r=R}dz=-2\nu\intop_{-a}^au_{,r}^2\bigg|_{r=0}^{r=R}dz\\
        &+\intop_{-a}^a\al(\t)f_0
u_{,r}r\bigg|_{r=R}dz  ,  \end{split}
\end{equation}
where in the last equality  we used \rf{1.9}$_1$ projected onto $S_1$.
Integration by parts in $r$ in the first term on the r.h.s. of \rf{5.7} gives 
\[
\intop_{-a}
\al f_0u_{,r}r\bigg|_{r=R}dz-\io\al f_0u_{,rr}dx-\io\al f_0 u_{,r}drdz,
\]
where we used \rf{5.8} again to note that $u_ {,r}|_{r=0} =0.$ We note again that the first term above cancells with the last term of \rf{5.9}, while the remaining terms can be estimated using the Young inequality by
\[
\fr\nu 4 \lp{u_{,rr}}{2,\O}2+\nu \ll|\fr{u_{,r}}r\rr|_{2,\O}^2+\fr c \nu \lp{\al f_0}{2,\O}2.
\]
Using the above estimates in \rf{5.7} and simplifying we get.
\begin{equation}\label{5.10}
    \begin{split}
    \fr12 &\dt \lp{u_{,r}}{2,\O}2+\fr \nu 4(\lp{u_{,rr}}{2,\O}2+\lp{u_{,rz}}{2,\O}2)\le2\nu \io\fr{u_{,r}^2}{r^2}dx\\
  &  +\io \fr{v_r}r\fr{u_{,r}}rudx+c \lp{u}{\infty,\O}2(\lp{v_{r,r}}{2,\O}2+\lp{v_{r,r}}{2,\O}2)+\fr c \nu \lp{\al f_0}{2,\O}2.
\end{split}
\end{equation}
Integrating \rf{5.10} with respect to time yields \begin{equation}\label{5.11}
\begin{split}
     &\lp{u_{,r}(t)}{2,\O}{2}+\nu(\lp{u_{,rr}}{2,\Ot}{2}+\lp{u_{,rz}}{2,\Ot}{2})\\
    &\le cD_1^2(1+D_2)+cD_1^2D_2^2+\phi(\t_*,\t^*)\lp{f_0}{2,\Ot}{2}\\
     &+\lp{u_{,r(0)}}{2,\O}{2}\equiv  cD_5^2.
\end{split}
   \end{equation}
This implies \rf{5.3} and concludes the proof.

\end{proof}
\section{Auxiliary estimates}\label{s6}
\begin{lemma}\lb{lemma 6.1}
    Any regular solution to \rf{1.1}-\rf{1.4} satisfies
    \begin{equation}\lb{6.1}
        \begin{split}
            &\hp{\o_{r}}{V(\Ot)}2 +\hp{\o_{z}}{V(\Ot)}2+\lp{\fr{\o_r}r}{2,\Ot}2\\ 
            &\le \fr1\nu\phi(D_1,D_2,D_4,D_5)\ll(\fr{R^{\ve_0}}{\ve_0}\lp{v_\phi}{\iy,\Ot} {\ve_0}+  \fr{R^{2{\ve_0}}}{{\ve_0}^2}\lp{v_\phi}{\iy,\Ot} {2{\ve_0}}\rr)\lp{\nb \G}{2,\Ot}{}\\
            &+ c D_9^2 \lp{\nb \t}{2,\Ot}2+ c D_{10}^2,
        \end{split}
    \end{equation}
    where
    \begin{equation}\lb{6.2}
        \begin{split}
            &D_9^2=\fr{\phi(\t_*,\t^*)}{\nu}\lp{f_\vp}{3,\iy,\Ot}2\\
            &D_{10}^2=(D_4+D_5)\hp{f_\vp}{L_2(0,t;L_3(S_1)}{}\\
            &+\fr1\nu (\lp{F_r}{6/5,2/\Ot}2+\lp{F_z}{6/5,2/\Ot}2)+\lp{\o_r(0)}{2,\O}2+\lp{\o_z(0)}{2,\O}2.
        \end{split}
    \end{equation}
\end{lemma}
\begin{proof}
    Multiplying \rf{1.8}$_1$ by $\o_r$, \rf{1.8}$_3$ by $\o_z$, adding the resulting equations and integrating over $\Ot$, we obtain
\begin{equation}\lb{6.3}
    \begin{split}
        \fr12&(\lp{\o_r(t)}{2,\O}2+\lp{\o_r(t)}{2,\O}2)\\
        &+\nu (\lp{\nb \o_r}{2,\Ot}2+\lp{\nb \o_z}{2,\Ot}2+\lp{\fr{\o_r}{r}}{2,\Ot}2)\\
        &=\nu \intop_{S^t}(\bar n\cd \nb\o_z\o_z+\bar n\cd \nb\o_r\o_r)dSdt'\\
        &+\iot (v_{r,r}\o_r^2+v_{z,z}\o_z^2+(v_{r,z}+v_{z,r})\o_r\o_z)dxdt'\\
        &+\iot \dot \al (-\t_{,z}+\fr1r(r\t)_{,r})f_\vp dxdt'+\iot (F_r\o_r+F_z\o_z)dxdt'\\
        &+\fr12 (\lp{\o_r(0)}{2,\O}2+\lp{\o_z(0)}{2,\O}2)\equiv I_1+J+I_2+I_3\\
        &+\fr12(\lp{\o_r(0)}{2,\O}2+\lp{\o_z(0)}{2,\O}2).
         \end{split}
\end{equation}
First we examine $I_{1}$.  Since $ \o _r=-v_{\vp ,z}, \,\vphi|_{r=R} =0$ and $\left.v_{\varphi, z}\right|_{S_{2}}=0$ we obtain
$$
\int_{S} \bar{n} \cdot \nabla \omega_{r} \omega_{r} d S=0 .
$$
Using \rf{1.5}$_{3}$ we get $\o_z=v_{\vp,r}+\fr{\vphi}{r}  $. Then

\begin{multline*}
 -\nu \intop_{S_{1}^{t}} \bar{n} \cdot\nabla \o_z \o_z d S_{1} d t^{\prime}=-\nu \intop_{S_{1}^{t}} \partial_{r}\left(v_{\varphi, r}+\frac{\vphi}{r}\right)\left(v_{\varphi, r}+\frac{\vphi}{r}\right) R d z d t^{\prime} \\
 =-\nu \intop_{0}^{t} \intop_{-a}^{a}\left(v_{\vp, r r}+\frac{v_{\vp, r}}{r}\right) v_{\vp, r}\bigg|_{r=R} R d z d t^{\prime} \equiv I_{1}^1,
\end{multline*}
where we used that $\left.v_{\varphi}\right|_{S_{1}}=0$.
Projecting $(1.7)_{2}$ on $S_{1}$ yields
$$
-\nu\left(v_{\varphi, rr}+\frac{1}{r} v_{\varphi, r}\right)=\alpha(\theta) f_{\varphi} .
$$
Hence
$$
\begin{aligned}
& I_{1}^{1}=R \intop_{0}^{t} \intop_{-a}^{a} \alpha(\theta) f_{\varphi} v_{\varphi, r}\bigg|_{r=R} d z d t^{\prime}=\al(\t_*,\t^*)\intop_{0}^{t} \intop_{-a}^{a}f_{\varphi}(u_{, r}-\frac{1}{R} u) d z d t^{\prime}
\end{aligned}
$$
and
\begin{align*}
\left|I_{1}^{1}\right| &\leqslant c \alpha\left(\theta_{*}, \theta^{*}\right)\left|f_{\varphi}\right|_{2, S_{1}^{t}}\left(|u_{,r}|_{2, S_{1}^{t}}+|u|_{2, S_{1}^{t}}\right)\\  & \leqslant c \alpha\left(\theta_{*}, \theta^{*}\right)(D_4+D_5).
\end{align*}
Finally,
$$
-\nu \intop_{S_{2}^{t}} \bar{n}\cdot \nabla \omega_{z} \omega_{z} d S_{2} d t^{\prime}=-\nu \intop_{S_{2}^{t}} \frac{1}{r} u_{,zr}  \frac{1}{r} u_{, r} d S_{2} dt'=0
$$
Summarizing,
\begin{equation}\lb{6.4}
    \quad I_{1} \leq c \alpha\left(\theta_{*}, \theta^{*}\right)\left|f_{\varphi}\right|_{2, S_{1}^{t}}\left(D_{4}+D_{5}\right).
\end{equation}
Next, we examine $I_{2}, I_{3}$. By the H\"older inequality, we get
\begin{equation}\lb{6.5}
    \begin{aligned}
I_{2} & \leq c \alpha\left(\theta_*, \theta^{*}\right)\|\theta\|_{1,2, \Ot}\left|f_{\vp}\right|_{2, \Ot} \\
& \leq c \alpha\left(\theta_{*}, \theta^{*}\right) D_{0}\left|f_{\varphi}\right|_{2, \Ot}, \\
I_{3} & \leq \varepsilon\left(\left|\omega_{r}\right|_{6,2, \Ot}^2+\left|\omega_{z}\right|_{6,2, \Omega^{t}}^{2}\right) \\
& +\frac{1}{4 \varepsilon}\left(\left|F_{r}\right|_{6 / 5,2, \Ot}^{2}+\left|F_{z}\right|_{6 / 5,2, \Omega^{t}}^{2}\right).
\end{aligned}
\end{equation}
Finally, we examine
\begin{equation}\lb{6.6}
    J=\intop_{\Omega^{t}}\left[v_{r, r} \omega_{r}^{2}+v_{z, z} \omega_{z}^{2}+\left(v_{r, z}+v_{z, r}\right) \omega_{r} \omega_{z}\right] d x d t^{\prime}.
\end{equation} 
Using \rf{1.5} and \rf{1.15} yields
\begin{equation}\lb{6.7}
    \begin{aligned}
J&=\intop_{\Omega^{t}}\left[-\psi_{,z r}\left(\frac{1}{r} u_{,z}\right)^{2}+\left(\psi_{,rz}+\frac{\psi_{,z}}{r}\right)\left(\frac{1}{r} u_{, r}\right)^{2}\right. \\
& \left.-\left(-\psi_{,zz}+\psi_{, r r}+\frac{1}{r} \psi_{, r}-\frac{\psi}{r^{2}}\right)\left(\frac{1}{r} u_{,z}\right)\left(\frac{1}{r} u_{, r}\right)\right] d x d t^{\prime} \\
& \equiv J_{1}+J_{2}+J_{3} \cdot
\end{aligned}
\end{equation}
Consider $J_{1}$. Integrating by parts with respect to $z$ and using that \linebreak $\left.u_{,z}\right|_{S_{2}}=0$, we obtain
\begin{multline*}
J_{1}=  -\intop_{\Omega^{t}} \psi_{,z r} \frac{1}{r} u_{,z} \frac{1}{r} u_{,z} d x d t^{\prime}=\intop_{\Omega^{t}} \psi_{,zz r} \frac{1}{r^{2}} u_{,z} u d x d t^{\prime} \\
 +\intop_{\Omega^{t}} \psi_{,z r} \frac{1}{r^{2}} u_{,zz} u d x d t^{\prime} \equiv J_{11}+J_{12}
\end{multline*} 
Using the transformation $\psi=r \psi_{1}$, we have
$$
J_{11}=\intop_{\Omega^{t}}\left(\frac{\Psi_{1,zz}}{r}+\psi_{1,zzr}\right) \frac{u_{,z}}{r} u d x d t^{\prime} \equiv J_{11}^{1}+J_{11}^{2} .
$$
By the H\"older inequality,
\begin{multline*}
\left|J_{11}^{1}\right| \leq \intop_{\Omega^{t}}\left|\frac{\psi_{1,zz}}{r^{1-\varepsilon_{0}} }\right|\left|\frac{u_{,z}}{r}\right|\left|v_{\varphi}\right|^{\varepsilon_{0}}|u|^{1-\varepsilon_{0}} d x d t^{\prime}\\
\leq c \alpha\left(\theta_{*},\theta^{*}\right) D_{2}^{1-\varepsilon_{0}}\left|v_{\varphi}\right|_{\iy, \Omega^{t}}^{\varepsilon_{0}}\left|\frac{u_{,z}}{r}\right|_{2, \Omega^{t}} \lp{\fr{\Psi_{1,zz}} {r^{1-\ve_0}}}{2,\Ot}{},
\end{multline*}
where \rf{2.16} is used. In view of \rf{2.15}
$$
\left|\frac{u_{,z}}{r}\right|_{2, \Ot} \leq\left|v_{\varphi, z}\right|_{2, \Ot} \leq D_{1}
$$
and \rf{2.19},\rf{3.13} imply
$$
\left|\frac{\psi_{1,zz}}{r^{1-\varepsilon_{0}}}\right| \leqslant c \frac{R^{\varepsilon_{0}}}{\varepsilon_{0}}\left|\psi_{1,zz r}\right|_{2, \Ot} \leqslant c \frac{R^{\varepsilon_{0}}}{\varepsilon_{0}}\left|\Gamma_{,z}\right|_{2, \Ot}.
$$
Summarizing,
$$
\left|J_{11}^{1}\right| \leqslant c \alpha\left(\theta_{*}, \theta^{*}\right) \frac{R^{\varepsilon_{0}}}{\varepsilon_{0}} D_{1} D_{2}^{1-\varepsilon_{0}}\lp{v_{\vp}}{\iy, \Ot}{{\varepsilon_{0}}}\lp{\G_{,z}}{2, \Ot}{} .
$$
Next,
$$
\begin{aligned}
\left|J_{11}^{2}\right| & \leq|u|_{\infty, \Ot}\left|\frac{u_{,z}}{r}\right|_{2, \Ot}\left|\psi_{1,zz r}\right|_{2, \Omega^{t}} \\
& \leqslant c \alpha\left(\theta_{* }, \theta^{*}\right) D_{2} D_{1}\left|\G_{,z}\right|_{2, \Ot}
\end{aligned}
$$
where \rf{2.15},\rf{2.19},\rf{3.13}  were used.
Hence,
\begin{multline}\lb{6.8}
\left|J_{1 1}\right|  \leqslant c \alpha\left(\theta_{*}, \theta^{*}\right) D_{1} D_{2}^{1-\varepsilon_{0}}\left|v_{\varphi}\right|_{\infty, \Ot}^{\varepsilon_{0}}\left|\Gamma_{,z}\right|_{2, \Omega t} \\
+  c \alpha\left(\theta_{*}, \theta^{*}\right) D_{1} D_{2}\left|\Gamma_{,z}\right|_{2, \Omega^{t}}.
\end{multline}
Next,
$$
J_{12}=\intop_{\Omega^{t}}\left(\frac{\psi_{1,z}}{r^{2}}+\frac{\psi_{1,z r}}{r}\right) u_{,zz} u dx dt \equiv J_{12}^{1}+J_{12}^{2} .
$$
Estimates \rf{2.16},\rf{3.24},\rf{5.1} imply
$$
\begin{aligned}
\left|J_{12}^{2}\right| & \leq|u|_{\iy, \Ot}\left|u_{,zz}\right|_{2, \Omega^{t}}\left|\frac{\psi_{1,z r}}{r}\right|_{2, \Ot} \\
& \leq c \alpha\left(\theta_{* }, \theta^{*}\right) D_{2} D_{4}\left|\G_{,z}\right|_{2, \Ot} .
\end{aligned}
$$
Hence,
\begin{equation}\lb{6.9}
|J_{12}| \leqslant\bigg|\intop_{\Omega^{t}} \frac{\psi_{1,z}}{r^{2}} u_{,zz} u d x d t^{\prime}\bigg| 
 +c \alpha\left(\theta_{*}, \theta^{*}\right) D_{2} D_{4}\left|\G_{,z}\right|_{2, \Ot}.
\end{equation}
Definition of $J_{1}$ and \rf{6.8},\rf{6.9}  imply
\begin{multline}\lb{6.10}
\left|J_{1}\right| \leqslant c \alpha\left(\theta_{* }, \theta^{*}\right)\left[\frac{R^{\varepsilon_{0}}}{\varepsilon_{0}} D_{1} D_{2}^{1-\varepsilon_{0}}\left|v_{\vp}\right|_{\iy, \Ot}^{\varepsilon_{0}} 
 +D_{1} D_{2}+D_{2} D_{4}\right]\left|\G_{,z}\right|_{2, \Omega t}\\+\left|\intop_{\Omega^{t}} \frac{\psi_{1,z}}{r^{2}} u_{,zz} u d x d t\right| .
\end{multline}

Next, we estimate $J_{2}.$  We can write it in the form
$$
J_{2}=\intop_{\Omega^{t}}\left(\psi_{, r z}+\frac{\psi_{,z}}{r}\right) \frac{1}{r} u_{, r} u_{, r} d r d z d t^{\prime} .
$$
Integrating by parts with respect to $r$ yields
$$
\begin{aligned}
& J_{2}=\intop_{0}^{t} \intop_{-a}^{a}\left(\psi_{,r z}+\frac{\psi_{,z}}{r}\right) \frac{1}{r} u_{, r} u\bigg|_{r=0} ^{r=R} d z d t^{\prime} \\
& -\intop_{\Omega^{t}}\left(\psi_{,rz}+\frac{\psi_{,z}}{r}\right)_{, r} \frac{1}{r} u_{, r} u d r d z d t^{\prime} \\
& -\intop_{\Omega^{t}}\left(\psi_{,rz}+\frac{\psi_{,z}}{r}\right)\left(\frac{1}{r} u_{1 r}\right)_{,r} \frac{u}{r} d x d t \\
&\equiv J_{20}+J_{21}+J_{22},
\end{aligned}
$$
where the boundary term vanishes because $\left.u\right|_{r=R}=0$ and \rf{1.20},\rf{1.21},    \rf{1.22} imply

\begin{align*}
\left.\left(\psi_{,rz}+\frac{\psi_{,z}}{r}\right) \frac{1}{r} u_{, r} u\right|_{r=0}&=\left.2 a_{1,z}\left(v_{\vp,r} r+\frac{v_{\vp}}{r}\right) r v_{\varphi}\right|_{r=0} \\
& =\left.4 a_{1,z} b_{1} r^{2} b_{1}\right|_{r=0}=0.
\end{align*}
Using the transformation $\psi=r \psi_{1}$ in $J_{21}$ yields
$$
\begin{aligned}
J_{21} & =-\intop_{\Omega^{t}}\left(2 \psi_{1,z}+r \psi_{1, r z}\right) _{,r}\frac{1}{r} \frac{1}{r} u_{, r} u d x d t^{\prime} \\
& =-\intop_{\Omega^{t}}\left(3 \psi_{1, r z}+r \psi_{1, r r z}\right) \frac{1}{r} \frac{1}{r} u_{, r} u d x d t^{\prime}.
\end{aligned}
$$
By the H\"older inequality, we have
$$
\begin{aligned}
\left|J_{21}\right| & \leq 3\left|\frac{\psi_{1, rz}}{r}\right|_{2, \Ot}\left|\frac{1}{r} u_{, r}\right|_{2, \Ot}|u|_{\iy, \Ot} \\
& +\left|\psi_{1, r r z}\right|_{2, \Ot}\left|\frac{1}{r} u_{1, r}\right|_{2, \Ot}+|u|_{\iy, \Ot} \\
& \leq c \alpha\left(\theta_{*}, \theta^{*}\right) D_{1} D_{2}\left|\Gamma_{,z}\right|_{2, \Omega t},
\end{aligned}
$$
where \rf{2.15},\rf{2.16},\rf{3.14} and \rf{3.24} were used.
Next, we consider $J_{22}$. Passing to variable $\psi_{1}$, we get
$$
J_{22}=-\intop_{\Ot}\left(2 \psi_{1,z}+r \psi_{1, r z}\right) \frac{1}{r}\left(\frac{1}{r} u_{, r}\right)_{,r} u d x d t^{\prime} .
$$

Hence
$$
\begin{aligned}
 \left|J_{22}\right|& \leq 2\left|\intop_{\Omega^{t}} \frac{\psi_{1,z}}{r}\left(\frac{1}{r} u_{, r}\right)_{, r} u d x d t^{\prime}\right|+\left|\intop_{\Omega^{t}} \psi_{1, r z}\left(\frac{1}{r} u_{, r}\right)_{,r} u d x d t^{\prime}\right| \\
& \equiv K_{1}+K_{2},
\end{aligned}
$$
where $K_{2}$ is bounded by
$$
\begin{aligned}
& K_{2} \leq\left|\intop_{\Omega^{t}} \frac{\psi_{1, r z}}{r} u_{, r r} u d x d t^{\prime}\right|+\left|\int_{\Omega^{t}} \frac{\psi_{1, r z}}{r} \frac{u_{, r}}{r} u d x d t^{\prime}\right| \\
& \equiv K_{2}^{1}+K_{2}^{2} .
\end{aligned}
$$

Using \rf{2.15},\rf{2.16}, \rf{5.2} and \rf{3.24}, we obtain
$$
\begin{aligned}
\left|K_{2}^{1}\right| & \leq\left|\frac{\psi_{1, rz}}{r}\right|_{2, \Ot}\left|u_{, r r}\right|_{2, \Ot}|u|_{\iy, \Ot} \\
& \leqslant c \alpha\left(\theta_{*}, \theta^{*}\right) D_{2} D_{5}\left|\Gamma_{,z}\right|_{2, \Ot} 
\end{aligned}
$$
and
$$
\begin{aligned}
\left|K_{2}^{2}\right| & \leq\left|\frac{\psi_{1, rz}}{r}\right|_{2, \Ot}\left|\frac{1}{r} u_{,r}\right|_{2, \Ot}|u|_{\iy, \Ot} \\
& \leq c \alpha\left(\theta_{*}, \theta^{*}\right) D_{1} D_{2}\left|\G_{,z}\right|_{2, \Ot} .
\end{aligned}
$$
Summarizing,
\begin{equation}\lb{6.11}
\begin{split}
|J_{2}|\leqslant c \alpha\left(\theta_{*}, \theta^{*}\right)\left[D_{1} D_{2}+D_{2} D_{5}\right] \lp{\G_{,z}}{2,\Ot}{} \\
 +2\bigg|\intop_{\Omega^{t}} \frac{\psi_{1,z}}{r}\left(\frac{1}{r} u_{, r}\right)_{,r} u d x d t^{\prime}\bigg| .
 \end{split}
\end{equation}
Finally, we examine $J_{3}$. Using that $\psi=r \psi_{1}$ yields
$$
J_{3}=-\intop_{\Omega^{t}}\left(-r \psi_{1,zz}+3\psi_{1, r}+ r\psi_{1, r r}\right) \frac{1}{r} u_{, r} \frac{1}{r} u_{,z} d x d t^{\prime},
$$
Integrating by parts with respect to $z$, using that $\left.\psi_{1}\right|_{S_{2}}=0$ and \linebreak $\left.\psi_{1,zz}\right|_{S_{2}}=-\left.\Gamma\right|_{S_{2}}=0$, we obtain
$$
\begin{aligned}
J_{3} & =\int_{\Omega^{t}}\left(-\psi_{1,zzz}+\frac{3}{r} \psi_{1, r z}+\psi_{1, r r z}\right) \frac{1}{r} u_{, r} u d x d t^{\prime} \\
& +\int_{\Omega^{t}}\left(-\psi_{1,zz}+\frac{3}{r} \psi_{1, r}+\psi_{1, r r}\right)\left(\frac{1}{r} u_{,r}\right)_{,z} u d x d t^{\prime} \\
& \equiv J_{31}+J_{32} .
\end{aligned}
$$
Using \rf{2.15},\rf{2.16},\rf{3.14} and \rf{3.24}, we have
$$
\left|J_{31}\right| \leq \alpha\left(\theta_{* 1} \theta^{*}\right) D_{1} D_{2}\left|\Gamma_{,z}\right|_{2, \Ot}.
$$
To estimate $J_{32}$ we recall that
$$
-\psi_{1, r r}-\frac{3}{r} \psi_{1, r}-\psi_{1,zz}=\Gamma .
$$
Then $J_{32}$ takes the form
$$
J_{32}=-\intop_{\Omega^{t}}\left(2 \psi_{1,zz}+\Gamma\right)\left(\frac{1}{r} u_{, r}\right)_{,z} u d x d t^{\prime} \equiv J_{32}^{1}+J_{32}^{2} .
$$
Continuing,
$$
\begin{aligned}
 \left|J_{32}^{1}\right|& \leq c\left|\intop_{\Omega^{t}} \frac{\psi_{1,zz}}{r^{1-\varepsilon_{0}}} u_{,rz} u^{1-\varepsilon_{0}} v_{\varphi}^{\varepsilon_{0}} d x d t\right| \\
& \leq c \alpha\left(\theta_{*}, \theta^{*}\right) D_{2}^{1-\varepsilon_{0}}\left|v_{\varphi}\right|_{\infty , \Omega}^{\varepsilon_{0}}\left|\frac{\psi_{1,zz}}{r^{1-\varepsilon_{0}}}\right|_{2, \Omega^{t}}\left|u_{, r z}\right|_{2, \Omega^{t}},
\end{aligned}
$$
where we used \rf{2.16}. Finally, we have
$$
\left|J_{32}^{1}\right| \leq c \alpha\left(\theta_{*}, \theta^{*}\right) D_{2}^{1-\varepsilon_{0}} D_{4} \frac{R^{\varepsilon_{0}}}{\varepsilon_{0}}\left|v_{\varphi}\right|_{\iy\Ot}^{\varepsilon_{0}}\left|\Gamma_{,z}\right|_{2, \Omega t^{t}}.
$$
Next,
$$
\begin{aligned}
\left|J_{32}^{2}\right| &\leq \intop_{\Omega^{t}}\left|\frac{\G}{r^{1-\varepsilon_{0}}}\right|\left|u_{,r z}\right||u|^{1-\varepsilon_{0}}\left|v_{\varphi}\right|^{\varepsilon_{0}} d x d t^{\prime} \\
& \leq c \alpha\left(\theta_{*}, \theta^{*}\right) D_{2}^{1-\varepsilon_{0}} D_{4} \frac{R^{\varepsilon_{0}}}{\varepsilon_{0}}\left|v_{\varphi}\right|_{\infty, \Omega^{t}}^{\varepsilon_{0}}\left|\Gamma_{,r}\right|_{2, \Omega^{t}}.
\end{aligned}
$$
Summarizing,
\begin{equation}\lb{6.12}
|J_{3}| \leq c \alpha\left(\theta_{*}, \theta^{*}\right)\left[D_{1} D_{2}\left|\Gamma_{,z}\right|_{2, \Ot}
+D_{2}^{1-\varepsilon_{0}} D_{4} \frac{R_{0}^{\varepsilon_{0}}}{\varepsilon_{0}}\left|v_{\vp}\right|_{\iy, \Ot}^{\varepsilon_{0}}|\nabla \Gamma|_{2, \Omega^{t}}\right].
\end{equation}
Using estimates \rf{6.10}, \rf{6.11}, \rf{6.13} in \rf{6.7} implies
\begin{equation}\lb{6.13}
    \begin{aligned}
|J|& \leq\left|\intop_{\Omega^{t}} \frac{\Psi_{1,z}}{r^{2}} u_{,zz} u d x d t\right|+2\left|\intop_{\Omega^{t}} \frac{\psi_{1,z}}{r}\left(\frac{1}{r} u_{, r}\right)_{,r} u d x d t\right|\\
& +c\al\left(\theta_{*}, \theta^{*}\right)\bigg[\frac{R^{\varepsilon_{0}}}{\varepsilon_{0}}\left(1+D_{1}\right) D_{2}^{1-\varepsilon_{0}}\lp{\vphi}{\iy,\Ot}{\ve_0}  \\
&+D_{1} D_{2}+D_{2} D_{4}+D_{2} D_{5}\bigg]
|\nabla \Gamma|_{2, \Ot} .
\end{aligned}
\end{equation}
Using estimates \rf{6.4},\rf{6.5},\rf{6.13} in \rf{6.3} implies the inequality
\begin{equation}\lb{6.14}
    \begin{aligned}
&|\omega_{r}(t)|_{2, \Omega}^{2}+\omega_{z}(t)|_{2, \Omega}^{2}+\nu\left(|\nb \omega_{r}\right|_{2, \Omega t}^{2}+\left|\nabla \omega_{z}\right|_{2, \Ot}^{2}+|\Phi|_{2, \Ot}^{2})\\
& \leq\left|\intop_{\Ot} \frac{\Psi_{1,z}}{r^{2}} u_{,zz} u d x d t^{\prime}\right|+2\bigg|\intop_{\Omega^{t}} \frac{\psi_{1,z}}{r}\left(\frac{1}{r} u_{, r}\right)_{, r} u d x d t^{\prime}\bigg|+c \alpha\left(\theta_{* }, \theta^{*}\right)\cdot \\
&  \qquad\cdot\bigg[\frac{R^{\varepsilon_{0}}}{\varepsilon_{0}}\left(1+D_{1}\right) D_{2}^{1-\varepsilon_{0}}\left|v_{\varphi}\right|_{\iy, \Ot}^{\varepsilon_{0}}+D_{1} D_{2} +D_{2}\left(D_{4}+D_{5}\right)\bigg]\cdot\\
& \qquad\cdot |\nb \Gamma|_{2, \Ot}+c| f_{\vp}|_{2, S_{1}^{t}}\left(D_{4}+D_{5}\right) +c \alpha\left(\theta_{*}, \theta^{*}\right) D_{0}\left|f_{\vp}\right|_{2, \Ot}\\
& \qquad+c\left(\left|F_{r}\right|_{6 / 5,2, \Ot}^{2}+\left|F_{z}\right|_{6 / 5,2, \Ot}^{2}\right) 
 +\left|\omega_{r}(0)\right|_{2, \Omega}^{2}+\left|\omega_{z}(0)\right|_{2, \Omega}^{2} .
\end{aligned}
\end{equation}
Recalling that $\omega_{r}=-\frac{1}{r} u_{,z}, \,\omega_{z}=\frac{1}{r} u_{, r}$(see \rf{1.15}) and applying the H\"older and Young inequalities to the first term on the r.h.s. of \rf{6.14} we bound it by
$$
\varepsilon_{1}\left|\omega_{r, z}\right|_{2, \Ot}^{2}+\frac{1}{4 \varepsilon_{1}} \intop_{\Omega^{t}} \frac{\psi_{1,z}^{2}}{r^{2}} u d x d t^{\prime} \equiv L_{1} .
$$
The second term in $L_{1}$ can be written in the form
\begin{equation*}
\intop_{\Omega^{t}} \frac{\psi_{1,z}^{2}}{r^{2\left(1-\varepsilon_{0}\right)}} \frac{u^{2}}{r^{2 \varepsilon_{0}}} d x d t^{\prime} \leq D_{2}^{2\left(1-\varepsilon_{0}\right)}\left|v_{\varphi}\right|_{\infty, \Omega^{t}}^{2 \varepsilon_{0}} \intop_{\Omega^{t}} \frac{\psi_{1,z}^{2}}{r^{2(1-\varepsilon_{0})}} d x d t^{\prime},
\end{equation*}
where $\varepsilon_{0}$ can be chosen as small as we need. By the Hardy inequality
$$
\intop_{\Omega^{t}} \frac{\psi_{1 ,z}^{2}}{r^{2\left(1-\varepsilon_{0}\right)}} d x d t^{\prime} \leq \frac{R^{2 \varepsilon_{0}}}{\varepsilon_{0}^{2}} \intop_{\Omega^{t}} \psi_{1, rz}^{2} d x d t^{\prime} \text {. }
$$
Applying the interpolation inequality \rf{2.21} (see $[$ BIN, Ch. 3, sect.15])
$$
\intop_{\Omega} \psi_{1, rz}^{2} d x \leqslant\left(\intop_{\Omega}\left|\nabla^{2} \psi_{1,z}\right|^{2} d x\right)^{\theta}\left(\intop_{\Omega} \psi_{1,z}^{2} d x\right)^{1-\theta},
$$
where $\theta$ satisfies the equality
$$
\frac{3}{2}-1=(1-\theta) \frac{3}{2}+\theta\left(\frac{3}{2}-2\right)\, \text { so } \,\theta=1 / 2 \text {. }
$$
Using \rf{2.18}, we get
$$
\intop_{\Ot} \psi_{1,zr}^{2} d x \les c\left|\nb^{2} \psi_{1,z}\right|_{2, \Omega}\left|\psi_{1,z}\right|_{2, \Omega} \les cD_{1}\left|\nb^{2} \psi_{1,z}\right|_{2, \Omega} .
$$
Summarizing,
\begin{equation}\lb{6.15}
    \begin{aligned}
& \left|\intop_{\Omega^{t}} \frac{\psi_{1,z}}{r^{2}} u_{,zz} u d x d t\right| \les \varepsilon_{1}\left|\omega_{r,z}\right|_{2, \Ot}^{2} \\
 &\quad + \frac{c}{4 \varepsilon_{1}} \alpha( \theta_{*}, \theta^{*}) D_{1} D_{2}^{2\left(1-\varepsilon_{0}\right)} \frac{R^{2 \varepsilon_{0}}}{\varepsilon_{0}^{2}}\left|v_{\varphi}\right|_{\iy, \Omega^{t}}^{2 \varepsilon_{0}}\left|\G_{,z}\right|_{2, \Omega^{t}}.
\end{aligned}
\end{equation}
Similarly, the second term on the r.h.s. of \rf{6.14} is bounded by
\begin{equation}\lb{6.16}
\begin{aligned}
& \left|\intop_{\Omega^{t}} \frac{\psi_{1,z}}{r}\left(\frac{1}{r} u_{, r}\right)_{, r} u d x d t^{\prime}\right| \les \varepsilon_{2}\left|\omega_{z, r}\right|_{2, \Ot}^{2} \\
+ & \frac{c}{4 \varepsilon_{2}} \alpha\left(\theta_{*}, \theta^{*}\right) D_{1} D_{2}^{2\left(1-\varepsilon_{0}\right)} \frac{R^{2 \varepsilon_{0}}}{\varepsilon_{0}^{2}}\left|v_{\varphi}\right|_{\iy, \Omega^{t}}^{2 \varepsilon_{0}}\left|\Gamma_{,z}\right|_{2, \Ot}.
\end{aligned}
\end{equation}
Using  \rf{6.15} and \rf{6.16} in \rf{6.14} yields \rf{6.1} .
\end{proof}
\begin{lemma}\label{lemma 6.2} Assume that $D_1,D_2$ are defined in Section \ref{ss2.4},  \linebreak  $v_\vp(0)\in L_\iy(\O),f_\vp/r\in L_1(0,t;L_\iy(\O)),\,\t_*\le \t\le\t^*$.
Then \begin{equation}\label{6.17}
\lp{v_\vp(t)}{\iy,\O}{}    \les\fr {D_2}{\sqrt{\nu}}D_1^{1/4}X^{3/4}+D_{11},
\end{equation}
where 
\begin{equation}\label{6.18}
D_{11}=D_2^{1/2}    \phi(\t_*,\t^*)\ll|\fr{f_\vp}r\rr|_{\iy,1,\Ot}^{1/2}+\lp{v_\vp(0)}{\iy,\O}{}.
\end{equation}

\end{lemma}
\begin{proof}
Multiplying \rf{1.7}$_2$ by $v_\vp|v_\vp|^{s-2}$ and integrating over $\O$  yields
\begin{equation}\lb{6.19}
    \begin{split}
        &\fr1s \dt \lp{v_\vp}{s,\O}s+\fr{4\nu (s-1)}{s^2}\ll|\nb|v_\vp|^{s/2}\rr|_{2,\O}^2+\nu\io\fr{|v_\vp|^{s/2}}{r^2}dx\\
    &=\io \ps_{,z}|v_\vp|^sdx+\io\al(\t)f_\vp v_\vp ^{s-2}dx
    \end{split}
\end{equation}
The first term on the r.h.s. of \rf{6.19} is estimated by
\[
\ve\io \fr{|v_\vp|^{s}}{r^2}dx+\fr{D_2^2}{4\ve}\io \ps_{,z}^2|v_\vp|^{s-2}dx
\]
The second integral on the r.h.s is estimated by 
\[\begin{split}
    &\io \al(\t)|f_\vp||v_\vp|^{s-1}dx=\io \al(\t)\ll|\fr{f_\vp}r\rr|r|v_\vp|^{s-1}dx\\
&\le D_2\phi(\t_*,\t^*)\io \ll|\fr{f_\vp}r\rr||v_\vp |^{s-2}dx\\
& \le D_2\phi(\t_*,\t^*)\ll|\fr{f_\vp}r\rr|_{s/2,\O}|v_\vp |_{s,\O}^{s-2}.
\end{split}
\]
    In view of the above estimates inequality \rf{6.19} reads
    \[
    \begin{split}
        \fr12 & \dt \lp {v_\vp} {s,\O}s\le\fr{D_2^2}{2\nu}\lp {\ps_{,z}}{s,\O}{s-2}\\
        &\le D_2\phi(\t_*,\t^*)\ll|\fr{f_\vp}r\rr|_{s/2,\O}|v_\vp |_{s,\O}^{s-2}.
    \end{split}
        \]
        Simplifying, we get 
        \[
        \dt \lp{v_\vp}{s,\O}2 \le \fr{D_2^2}{\nu}  \lp{\ps_{,z}}{s,\O}2 +2 D_2\phi(\t_*,\t^*)\ll|\fr{f_\vp}r\rr|_{s/2,\O}.
        \]
        Integrating the inequality with respect to time and passing with $s$ to $\iy$ we get
    \[\begin{split}
        & \lp{v_\vp}{\iy,\O}2\le \fr{D_2^2}\nu\intop_0^t \lp{\ps_{,z}}{\iy,\O}2 dt'+2 D_2\phi(\t_*,\t^*)\ll|\fr{f_\vp}r\rr|_{\iy,1,\Ot}\\
       &+\lp{v_\vp(0)}{\iy,\O}{2}.  
    \end{split}      
\]
Using the interpolation 
\[
|\ps_{,z}|_{\iy,\O}^2\le \lp{\ps_{,z}}{2,\O}{1/4}\lp{D^2\ps_{,z}}{{2,\O}}{3/4}
\]
and \rf{2.17} we obtain 
\[
\begin{split}
    \lp{v_\vp}{\iy,\O}2 &\le \fr{D_2^2}{\nu}D_1^{1/2}\lp{\G_{,z}}{2,\Ot}{3/2}+\phi(\t_*,\t^*)\ll|\fr{f_\vp}r\rr|_{\iy,1,\Ot}\\
    &+\lp{v_\vp}{\iy,\O}2.
\end{split}
\]
The above inequality implies \rf{6.18} and concludes the proof.
\end{proof}
  \begin{lemma} \label{lemma 6.3}
Assume that for any regular solution to \rf{1.1}-\rf{1.3} there exist positive constants  $c_0,c_{*}$ such that
\begin{equation}\lb{6.20}
    \fr{\lp{\vphi}{s,\iy,\Ot}{}}{\lp{\vphi}{\iy,\Ot}{}}\geqslant c_0,\qquad \fr{1}{\lp{\vphi}{\iy,\Ot}{}}\les c_*
\end{equation}
Then for $f_{\vp} \in L_{10 / 7}\left(\Omega^{t}\right), v_{\varphi}(0) \in L_{s}(\Omega), s \geqslant 2$   we have
\begin{equation}\lb{6.21}
    \fr12 \lp{\vphi}{s,\iy,\Ot}{}\les \fr{D_2^2 D_1^2}{c_0^{s-2}}+\fr{\lp{f_\vp}{10/7,\Ot}{}D_1}{c_0^{s-2}}+\fr12\lp{\vphi(0)}{s,\O}{}\\
    \equiv D_{12}.
\end{equation}
\end{lemma}
\begin{proof}
 Assume that for any given positive $c_{1}$, 
 \begin{equation}\lb{6.22}
     \lp{\vphi}{s,\iy,\Ot}{}\les c_1.
 \end{equation}
Then \rf{6.1} yields
\begin{equation}\lb{6.23}
    \lp{\Phi}{2,\Ot}{}\les \phi(D)\ll(\lp{\vphi}{\iy,\Ot}{2\dl}+1 \rr)+\phi(D)
\end{equation}
and \rf{6.11}  gives
\begin{equation}\lb{6.24}
    \lp{\vphi}{\iy,\Ot}{}\les \phi(D)(X^{3/4}+1).
\end{equation}
Using \rf{6.22}, \rf{6.23}, \rf{6.24} and \rf{4.2} in \rf{4.1} yields
\begin{equation}\lb{6.25}
    X^2\les \phi(D)(1+X^{\fr34\dl})c_1^\ve X^{2-\fr{\t}{2}}+\phi(D).
\end{equation}
Since $\delta$ is an arbitrary small the above inequality implies the estimate 
\begin{equation}\lb{6.26}
    X\les \phi(D).
\end{equation}
Since \rf{6.22} implies estimate\rf{6.26} easily we restrict our considerations to the case 
\begin{equation}\lb{6.27}
    \lp{\vphi}{s,\iy,\Ot}{}\geqslant c_1.
\end{equation}
Multiply \rf{1.7}$_{2}$ by $v_{\vp}\left|v_{\vp}\right|^{s-2}$ and integrate over $\Omega$. Then we obtain
\begin{equation}\lb{6.28}
    \begin{aligned}
 \frac{1}{s} \frac{d}{d t}&\left|v_{\vp}\right|_{s, \Omega}^{s}+\frac{4 \nu(s-1)}{s^{2}}\ll|{\nb \lp{\vphi}{}{s/2}}\rr|_{2,\O}^{2}
 +\nu \io \frac{\left|v_{\varphi}\right|^{s}}{r^{2}} d x \\
= & -\io \frac{v_{r}}{r} \lp{\vphi}{}{s} d x+\io f_{\varphi} \vphi\left|\vphi\right|^{s-2} d x .
\end{aligned}
\end{equation}
Hence,
\begin{equation*}
    \begin{aligned}
 \frac{1}{s} \frac{d}{d t}&\left|v_{\vp}\right|_{s, \Omega}^{s}+\frac{4 \nu(s-1)}{s^{2}}\ll|{\nb \lp{\vphi}{}{s/2}}\rr|_{2,\O}^{2}
 +\fr{\nu}2 \io \frac{\left|v_{\varphi}\right|^{s}}{r^{2}} d x \\
= & \io v_{r}^2\lp{\vphi}{}{s} d x+\io |f_{\varphi}| |\vphi|^{s-1} d x .
\end{aligned}
\end{equation*}
Continuing, we have
\[
\lp{\vphi}{s,\O}{s-2}\fr12\dt \lp{\vphi}{s,\O}{2}\les D_2^2 \lp{\vphi}{\iy,\O}{s-2}\io\fr{v_r^2}{r^2 }dx+\lp{\vphi}{\iy,\O}{s-2}\io |f_{\varphi}| |\vphi|^{s-1} d x .
\]Since \rf{6.27} holds we have
\begin{equation}\lb{6.29}
    \fr12 \dt \lp{\vphi}{s,\O}{2}\les\fr{D_2^2}{\bar f^{s-2}}\io\fr{v_r^2}{r^2 }dx+\fr1{\bar f^{s-2}}\io |f_{\varphi}| |\vphi| d x .
\end{equation}
where
\[
\bar f=\fr{|\vphi|}{\lp{\vphi}{\iy,\O}{}},\quad \bar f_s=\ll(\io\bigg|\fr{|\vphi|}{\lp{\vphi}{\iy,\O}{}}\bigg |^s dx  \rr)^{1/s}
\]

Integrating \rf{6.29} with respect to time yields
\begin{equation}\lb{6.30}
    \fr12\lp{\vphi(t)}{s,\O}{2}\les \fr{D_2^2D_1^2}{\inf_t {\bar f_s^{s-2}}}+\fr{\lp{f_\vp }{10/7,\Ot}{} D_1}{\inf_t {\bar f_s^{s-2}}}+\fr12 \lp{\vphi(0)}{s,\O}{2}.
\end{equation}
The above inequality implies \rf{6.21}.

Now we add some comments concerning condition \rf{6.20}. We have that $\bar{f}_{\infty}=1$, $\bar{f}_{s} \les|\Omega|^{1 / s}$, where $|\Omega|$ is the measure of $\Omega$. Assuming that $\bar{f} \in C^{\alpha, \alpha / 2}\left(\Omega ^T\right)$, for any $\varepsilon>0$ there exists a set $A\subset \Omega$ having positive measure $|A|$ such that $|\bar{f}(x, t)| \geqslant 1-\varepsilon$ if $x \in A$.
Having that $\bar{f} \in C^{\alpha, \alpha / 2}\left(\Omega^{T}\right)$ we have that for any $t \in(0, T), \bar{f} \in C^{\alpha}(\Omega)$ so the measure $|A|$ can be assumed as independent of $t$.

Hence
$$
\intop_{\Omega}|\bar{f}(x, t)|^{s} d x \geqslant \intop_{A}|\bar{f}(x, t)|^{s} d x \geqslant|A|(1-\varepsilon)^{s}
$$
This implies that
$$
\bar{f}_s \geqslant|A|^{1 / s}(1-\varepsilon)
$$
This is a motivation for \rf{6.20}. However, \rf{6.20} is not proved.
  
\end{proof}

\section{Global estimate for regular solutions}\label{sec_reg}

\begin{lemma}\label{lemma 7.1}
Assume that $f\in W_2^{2,1}(\Omega^t),\,g\in W_2^{2,1}(\Ot), v(0)\in W_2^3(\Omega),$ \linebreak $\t (0)\in W_2^3(\O )$. 

Then for $t$ sufficiently small there exists a local solution to  problem \linebreak \rf{1.1}-\rf{1.4} such that $v\in W_2^{4,2}(\Omega^t),\, \t\in W_2^{4,2}(\Ot),\,\nabla p\in W_2^{2,1}(\Omega^t)$ and
\begin{multline}\label{7.1}
\|v\|_{W_2^{4,2}(\Omega^t)} +\hp{\t}{ W_2^{4,2}(\Omega^t)}{}+\|\nabla p\|_{W_2^{2,1}(\Omega^t)}\\
\leq C(\|f\|_{W_2^{2,1}(\Omega^t)}+ \|g\|_{W_2^{2,1}(\Omega^t)}+\|v(0)\|_{W_2^3(\Omega)}+\|\t(0)\|_{W_2^3(\Omega)}).
\end{multline}
\end{lemma}

The proof is standard.

\subsection{Global estimate for regular solutions}
We first introduce some functional analytic tools to handle the anisotropic Sobolev spaces.

\begin{definition}[Anisotropic Sobolev and Sobolev-Slobodetskii spaces]\label{d7.2}
We denote by
\begin{enumerate}
\item[1.] $W_{p,p_0}^{k,k/2}(\Omega^T)$, $k,k/2\in\N\cup\{0\}$, $p,p_0\in[1,\infty]$ -- the anisotropic Sobolev space with a mixed norm, which is a completion of $C^\infty(\Omega^T)$-functions under the norm
$$
\|u\|_{W_{p,p_0}^{k,k/2}(\Omega^T)}=\left(\int_0^T\left(\sum_{|\alpha|+2\alpha\le k}\int_\Omega |D_x^\alpha\partial_t^au|^p\right)^{p_0/p}\d t\right)^{1/p_0}.
$$
\item[2.] $W_{p,p_0}^{s,s/2}(\Omega^T)$, $s\in\R_+$, $p,p_0\in[1,\infty)$ -- the Sobolev-Slobodetskii space with the finite norm
\[  
\begin{split}
&\|u\|_{W_{p,p_0}^{s,s/2}(\Omega^T)}=\sum_{|\alpha|+2a\le[s]}\|D_x^\alpha\partial_t^au\|_{L_{p,p_0}(\Omega^T)}\\
&\quad+\bigg[\intop_0^T\!\!\bigg(\intop_\Omega\!\intop_\Omega\!\sum_{|\alpha|+2a= [s]}\hskip-10pt{|D_x^\alpha\partial_t^au(x,t)-D_{x'}^\alpha\partial_t^au(x',t)|^p\over |x-x'|^{n+p(s-[s])}}\d x\d x'\bigg)^{p_0/p}\hskip-5pt \d t\bigg]^{1/p_0}\\
&\quad+\bigg[\intop_\Omega\!\!\bigg(\intop_0^T\!\intop_0^T\!\sum_{|\alpha|+2a=[s]}\hskip-10pt {|D_x^\alpha\partial_t^au(x,t)-D_x^\alpha\partial_{t'}^au(x,t')|^{p_0}\over|t-t'|^{1+p_0({s\over 2}-[{s\over 2}])}}\d t\d t'\bigg)^{p/p_0}\hskip-5pt \d x\bigg]^{1/p}, 
\end{split}\]
where $a\in\N\cup\{0\}$, $[s]$ is the integer part of $s$ and $D_x^\alpha$ denotes the partial derivative in the spatial variable $x$ corresponding to multiindex $\alpha$. For $s$ odd the last but one term in the above norm vanishes whereas for $s$ even the last two terms vanish. We also use notation $L_p(\Omega^T)=L_{p,p}(\Omega^T)$, $W_p^{s,s/2}(\Omega^T)=W_{p,p}^{s,s/2}(\Omega^T)$.
\item[3.] $B_{p,p_0}^l(\Omega)$, $l\in\R_+$, $p,p_0\in[1,\infty)$ -- the Besov space with the finite norm
$$
\|u\|_{B_{p,p_0}^l(\Omega)}=\|u\|_{L_p(\Omega)}+\bigg(\sum_{i=1}^n\intop_0^\infty {\|\Delta_i^m(h,\Omega)\partial_{x_i}^ku\|_{L_p(\Omega)}^{p_0}\over h^{1+(l-k)p_0}}\d h\bigg)^{1/p_0},
$$
where $k\in\N\cup\{0\}$, $m\in\N$, $m>l-k>0$, $\Delta_i^j(h,\Omega)u$, $j\in\N$, $h\in\R_+$ is the finite difference of the order $j$ of the function $u(x)$ with respect to $x_i$ with
\begin{equation*}
\begin{split}
&\Delta_i^1(h,\Omega)u=\Delta_i(h,\Omega)\\
&=u(x_1,\dots,x_{i-1},x_i+h,x_{i+1},\dots,x_n)-u(x_1,\dots,x_n),\\
&\Delta_i^j(h,\Omega)=\Delta_i(h,\Omega)\Delta_i^{j-1}(h,\Omega)u\quad {\rm and}\quad \Delta_i^j(h,\Omega)u=0\\
&{\rm for}\quad x+jh\not\in\Omega.
\end{split}
\end{equation*}
In has been proved in \cite{G} that the norms of the Besov space $B_{p,p_0}^l(\Omega)$ are equivalent for different $m$ and $k$ satisfying the condition $m>l-k>0$.
\end{enumerate}
\end{definition}

We need the following interpolation lemma.

\begin{lemma}[Anisotropic interpolation, see {\cite[Ch. 4, Sect. 18]{BIN}}]\label{lemma 7.3}
Let $u\in W_{p,p_0}^{s,s/2}(\Omega^T)$, $s\in\R_+$, $p,p_0\in[1,\infty]$, $\Omega\subset\R^3$. Let $\sigma\in\R_+\cup\{0\}$, and
$$
\varkappa={3\over p}+{2\over p_0}-{3\over q}-{2\over q_0}+|\alpha|+2a+\sigma<s.
$$
Then $D_x^\alpha\partial_t^au\in W_{q,q_0}^{\sigma,\sigma/2}(\Omega^T)$, $q\ge p$, $q_0\ge p_0$ and there exists $\varepsilon\in(0,1)$ such that
$$
\|D_x^\alpha\partial_t^au\|_{W_{q,q_0}^{\sigma,\sigma/2}(\Omega^T)}\le\varepsilon^{s-\varkappa} \|u\|_{W_{p,p_0}^{s,s/2}(\Omega^t)}+c\varepsilon^{-\varkappa}\|u\|_{L_{p,p_0}(\Omega^t)}.
$$
We recall from \cite{B} the trace and the inverse trace theorems for Sobolev spaces with a mixed norm.
\end{lemma}

\begin{lemma}\label{lemma 7.4}
(traces in $W_{p,p_0}^{s,s/2}(\Omega^T)$, see \cite{B})
\begin{itemize}
\item[(i)] Let $u\in W_{p,p_0}^{s,s/2}(\Omega^t)$, $s\in\R_+$, $p,p_0\in(1,\infty)$. Then $u(x,t_0)=u(x,t)|_{t=t_0}\!$ for $t_0\in[0,T]$ belongs to $B_{p,p_0}^{s-2/p_0}(\Omega)$, and
$$
\|u(\cdot,t_0)\|_{B_{p,p_0}^{s-2/p_0}(\Omega)}\le c\|u\|_{W_{p,p_0}^{s,s/2}(\Omega^T)},
$$
where $c$ does not depend on $u$.
\item[(ii)] For given $\tilde u\in B_{p,p_0}^{s-2/p_0}(\Omega)$, $s\in\R_+$, $s>2/p_0$, $p,p_0\in(1,\infty)$, there exists a function $u\in W_{p,p_0}^{s,s/2}(\Omega^t)$ such that $u|_{t=t_0}=\tilde u$ for $t_0\in[0,T]$ and
$$
\|u\|_{W_{p,p_0}^{s,s/2}(\Omega^T)}\le c\|\tilde u\|_{B_{p,p_0}^{s-2/p_0}(\Omega)},
$$
where constant $c$ does not depend on $\tilde u$.
\end{itemize}
\end{lemma}

We need the following imbeddings between Besov spaces

\begin{lemma}[see {\cite[Th. 4.6.1]{T}}]\label{lemma 7.5}
|Let $\Omega\subset\R^n$ be an arbitrary domain.
\begin{itemize}
\item[(a)] Let $s\in\R_+$, $\varepsilon>0$, $p\in(1,\infty)$, and $1\le q_1\le q_2\le\infty$. Then
$$
B_{p,\infty}^{s+\varepsilon}(\Omega)\subset B_{p,1}^{s+\varepsilon}(\Omega)\subset B_{p,q_2}^s(\Omega)\subset B_{p,q_1}^s(\Omega)\subset B_{p,\infty}^{s-\varepsilon}(\Omega)\subset B_{p,1}^{s-\varepsilon}(\Omega).
$$
\item[(b)] Let $\infty>q\ge p>1$, $1\le r\le\infty$, $0\le t\le s<\infty$ and
$$
t+{n\over p}-{n\over q}\le s.
$$
Then $B_{p,r}^s(\Omega)\subset B_{q,r}^t(\Omega)$.
\end{itemize}
\end{lemma}

\begin{lemma}[see {\cite[Ch. 4, Th. 18.8]{BIN}}]\label{lemma 7.6}
Let $1\le\theta_1<\theta_2\le\infty$. Then
$$
\|u\|_{B_{p,\theta_2}^l(\Omega)}\le c\|u\|_{B_{p,\theta_1}^l(\Omega)},
$$
where $c$ does not depend on $u$.
\end{lemma}

\begin{lemma}[see {\cite[Ch. 4, Th. 18.9]{BIN}}] \label{lemma 7.7}
Let $l\in\N$ and $\Omega$ satisfy the $l$-horn condition.\\
Then the following imbeddings hold
$$
\begin{aligned}
    &\|u\|_{B_{p,2}^l(\Omega)}\le c\|u\|_{W_p^l(\Omega)}\le c\|u\|_{B_{p,p}^l(\Omega)}, &&1\le p\le 2,\\
&\|u\|_{B_{p,p}^l(\Omega)}\le c\|u\|_{W_p^l(\Omega)}\le c\|u\|_{B_{p,2}^l(\Omega)}, &&2\le p<\infty,\\
&\|u\|_{B_{p,\infty}^l(\Omega)}\le c\|u\|_{W_p^l(\Omega)}\le c\|u\|_{B_{p,1}^l(\Omega)},& &1\le p\le\infty.
\end{aligned}
$$
\end{lemma}

Consider the nonstationary Stokes system in $\Omega\subset\R^3$:
\[
\begin{split}
v_t-\nu\Delta v+\nabla p&=f,\\
\div\, v&=0
\end{split}
\]
with the boundary conditions \eqref{1.2} and given initial condition $v(0)$.

\begin{lemma}[see \cite{MS}]\label{lemma 7.8}
Assume that $f\in L_{q,r}(\Omega^T)$, $v(0)\in B_{q,r}^{2-2/r}(\Omega)$, $r,q\in(1,\infty)$. Then there exists a unique solution to the above system such that $v\in W_{q,r}^{2,1}(\Omega^T)$, $\nabla p\in L_{q,r}(\Omega^T)$ with the following estimate in a Sobolev spaces with a mixed norm
\begin{equation}
\|v\|_{W_{q,r}^{2,1}(\Omega^T)}+\|\nabla p\|_{L_{q,r}(\Omega^T)}\le c(\|f\|_{L_{q,r}(\Omega^T)}+\|v(0)\|_{B_{r,q}^{2-2/r}(\Omega)}).
\label{7.2}
\end{equation}
\end{lemma}

\begin{proof}
Proof of Theorem \ref{th1.3}

Let us recall our problem

\begin{equation}\lb{7.3}
    \begin{array}{ll}
v_{, t}-\nu  \Delta v+\nabla  {p}=-v'\cdot \nb v+\alpha(\theta) f & \text { in } \Omega^{T}, \\
\operatorname{div} v=0 & \text { in } \Omega^{T}, \\
v_{r}=v_{\varphi}=\omega_{\varphi}=0 & \text { in } S_{1}^{T}, \\
v_{z}=\omega_{\vp}=v_{\varphi, z}=0 & \text { on } S_{2}^T, \\
v|_{t=0}=v(0) \equiv v_{0} & \text { in } \Omega,
\end{array}
\end{equation}
where $v^{\prime}=\left(v_{r}, v_{z}\right)$ and
\begin{equation}\lb{7.4}
 \begin{aligned}
& \theta_{,t}-\kappa \Delta \theta=-v' \cdot \nb \theta+g &&\text { in } \Omega^{T}\\
& \bar{n} \cdot \nabla \theta=0 &&\text { on } S_{1}^{T} \\
& \theta |_{t=0}=\theta(0) \equiv \theta_{0}&&\text { in } \Omega
\end{aligned}
\end{equation}
From \rf{1.26} we have
\begin{equation}\lb{7.5}
    \|\Phi\|_{V\left(\Omega^{t}\right)}+\|\Gamma\|_{V(\Ot))} \leqslant \phi\left(D_{1}, \cdots, D_{12}\right) \equiv \phi_{1}.
\end{equation}
Then Lemma \ref{l3.1} gives
\begin{equation}\lb{7.6}
    \left\|\psi_{1}\right\|_{2,\iy ,\Ot, } \leqslant c\|\Gamma\|_{V(\Ot)} \leqslant c \phi_{1} .
\end{equation}
From \rf{1.18} the following relations hold
\begin{equation}\lb{7.7}
   v _{r}=-r \psi_{1,z},\; v_{z}=2\psi_{1}+r \psi_{1, r}
\end{equation}
Hence, \rf{7.6},\rf{7.7}  and finite $R$ yield
\begin{equation}
    \left\|v_{r}\right\|_{1,2,\iy,\Ot}+ \left\|v_{z}\right\|_{1,2,\iy,\Ot} \leqslant c \phi_{1}
\end{equation}
The above inequality implies
\begin{equation}\lb{7.9}
    \lp{v'}{6,\iy,\Ot}{}\le c\phi_1
\end{equation}
The following energy type estimate for solutions to \rf{7.3} and \rf{7.4}  holds
\begin{equation}\lb{7.10}
    \hp{v}{1,2,\Ot}{}+ \hp{\t}{1,2,\Ot}{}\le d_1\equiv d_1(D_0,D_1)
\end{equation}
Estimates \rf{7.9} and \rf{7.10} imply
\begin{equation}\lb{7.11}
    |v'\cdot\nb v|_{3/2,2,\Ot}+ |v'\cdot\nb \t|_{3/2,2,\Ot}\le\phi_1d_1
\end{equation}
In view of \rf{7.11} and Lemma \ref{lemma 7.8}, we obtain 
\begin{equation}\lb{7.12}
    \begin{aligned}
& \|v\|_{W_{3/2,2}^{2,1}(\Ot)}+\|\t\|_{W_{3/2,2}^{2,1}(\Ot)}+|\nb p|_{3 / 2,2, \Omega^{t}} \\
&\leq c\left(\alpha (\theta_{*}, \t^{*}\right)
 |f|_{3 / 2, \Ot}+|g|_{3 / 2,2, \Ot}+\left\|v_{0}\right\|_{B_{3 / 2, 2}^1(\Omega)}\\
 &+\left\|\t_0\right\|_{B_{3 / 2, 2}^1(\Omega)} 
+\phi_{1} d_{1}) \equiv d_{2} .
\end{aligned}
\end{equation}
In view of the imbeddings (see [BIN, Ch. 3 , Sect.10])
\begin{equation}\lb{7.13}
    \begin{aligned}
& |\nabla v|_{5 / 2, \Omega^{t}} \leq c\|v\|_{W_{\frac{3}{2},{2}}^{2,1}(\Omega^{t})} \\
& |\nabla \t|_{5 / 2, \Omega^{t}} \leq c\|\t\|_{W_{\frac{3}{2},{2}}^{2,1}(\Omega^{t})} 
\end{aligned}
\end{equation}
and \rf{7.9}  we derive that
\begin{equation}\lb{7.14}
    |v'\cdot \nabla v|_{\frac{30}{17}, \frac{5}{2}, \Omega t}+\left|v' \cdot \nabla \theta\right|_{\frac{30}{17}, \frac{5}{2}, \Ot} \leq c \phi_{1} d_{2}.
\end{equation}
Applying again Lemma \ref{lemma 7.8} to problems \rf{7.3} and \rf{7.4} yields
\begin{equation}\lb{7.15}
\begin{aligned}
&\hp{v}{W_{\fr{30}{17},\fr52}^{2,1}(\Ot)}{}+\hp{\t}{W_{\fr{30}{17},\fr52}^{2,1}(\Ot)}{}+\lp{\nb p}{\fr{30}{17},\Ot}{}\\
&\leq c\left(\alpha (\theta_{*}, \t^{*}\right)
 |f|_{\fr{30}{17},\fr52, \Ot}+|g|_{\fr{30}{17},\fr52, \Ot}\\
 &+\left\|v_{0}\right\|_{B_{\fr{30}{17},\fr52}^{2-\fr45}(\Omega)}+\left\|\t_{0}\right\|_{B_{\fr{30}{17},\fr52}^{2-\fr45}(\Omega)} 
+\phi_{1} d_{2}) \equiv d_{3} .
\end{aligned}
\end{equation}
In view of the imbeddings (see [BIN, Ch. 3, Sect .10])
\begin{equation}\lb{7.16}
    \begin{aligned}
& |\nb v|_{\frac{10}{3}, \O t} \leq c\hp{v}{W_{\fr{30}{17},\fr52}^{2,1}(\Ot)}{}, \\
& |\nb \t|_{\frac{10}{3}, \O t} \leq c\hp{\t}{W_{\fr{30}{17},\fr52}^{2,1}(\Ot)}{}
\end{aligned}
\end{equation}
and \rf{7.9}, we have
\begin{equation}\lb{7.17}
    \left|v'\cdot \nb v\right|_{\frac{15}{7}, \frac{10}{3}, \Omega^{t}}+\left|v'\cdot \nb \t\right|_{\frac{15}{7}, \frac{10}{3}, \Omega^{t}} \leqslant c \phi_{1} d_{3}.
\end{equation}
Applying Lemma \ref{lemma 7.8} to problems \rf{7.3}, \rf{7.4} and using \rf{7.17} imply
\begin{equation}\lb{7.18}
\begin{aligned}
&\hp{v}{W_{\fr{15}{7},\fr{10}3}^{2,1}(\Ot)}{}+\hp{\t}{W_{\fr{15}{7},\fr{10}3}^{2,1}(\Ot)}{}+\lp{\nb p}{\fr{15}{7}\fr{10}3,\Ot}{}\\
&\leq c\left(\alpha (\theta_{*}, \t^{*}\right)
 \lp{f}{\fr{15}{7},\fr{10}3,\Ot}{}+\lp{f}{\fr{15}{7},\fr{10}3,\Ot}{}\\
 &+\left\|v_{0}\right\|_{B_{\fr{15}{7},\fr{10}{3}}^{2-\fr6{10}}(\Omega)}+\left\|\t_{0}\right\|_{B_{\fr{15}{7},\fr{10}{3}}^{2-\fr6{10}}(\Omega)}
+\phi_{1} d_{3}) \equiv d_{4} .
\end{aligned}
\end{equation}
Lemma \ref{lemma 7.4}  gives
\begin{equation}\lb{7.19}
    \begin{aligned}
         &\left\|v\right\|_{L_\iy (0,t;B_{\fr{15}{7},\fr{10}{3}}^{2-\fr6{10}}(\Omega))}\le c \hp{v}{W_{\fr{15}{7},\fr{10}3}^{2,1}(\Ot)}{},\\
         &\left\|\t \right\|_{L_\iy (0,t;B_{\fr{15}{7},\fr{10}{3}}^{2-\fr6{10}}(\Omega))}\le c \hp{\t}{W_{\fr{15}{7},\fr{10}3}^{2,1}(\Ot)}{}.
    \end{aligned}
\end{equation}
Theorem 18.10 from \cite{BIN} gives
\begin{equation}\lb{7.20}
    |v(t)|_{q, \Omega} \leqslant c\|v\|_{B_{\frac{15}{7}, \frac{10}{3}}^{7 /5}(\Omega)}
\end{equation}
which holds for any finite $q$ satisfying the relation $7 / 5 \geqslant 7 / 5-3 / q$.
Next, we use the imbeddings (see [BIN, Ch, 3, Sect .10])
\begin{equation}\lb{7.21}
    \begin{aligned}
& |\nabla v|_{5, \Omega^{t}} \leq c\|v\|_{W_{\frac{15}{7}, \frac{10}{3}}^{2,1}\left(\Omega^{t}\right)}, \\
& |\nabla \t|_{5, \Omega^{t}} \leq c\|\t\|_{W_{\frac{15}{7}, \frac{10}{3}}^{2,1}\left(\Omega^{t}\right)} .
\end{aligned}
\end{equation}
Estimates (7.20) and (7.21) imply
\begin{equation}\lb{7.22}
    \begin{aligned}
& \left|v' \cdot \nb v\right|_{5', \Ot} \leq c d_{4}^{2}, \\
& \left|v' \cdot \nb \t\right|_{5', \Ot} \leq c d_{4}^{2}.
\end{aligned}
\end{equation}
where $5'<5$ but it is arbitrary close to 5. In view of \rf{7.22} and Lemma \ref{lemma 7.8} we have
\begin{equation}\lb{7.23}
    \begin{aligned}
& \|v\|_{W_{5'}^{2,1}\left(\Omega^{t}\right)}+\|\theta\|_{W_{5'}^{2,1}\left(\Omega^{t}\right)}+|\nabla p|_{5', \Omega^{t}} \\
& \leq c\left(\al\left(\theta_{*}, \theta^{*}\right)|f|_{5', \Omega^{t}}+|g|_{5', \Omega^{t}}+\left\|v_{0}\right\|_{W_{5'}^{2-2 / 5'}(\O)}\right. \\
& \left.+\left\|\theta_{0}\right\|_{W_{5'}^{2-2 / 5'}(\O)}+d_{4}^{2}\right) \equiv d_{5} .
\end{aligned}
\end{equation}
From \rf{7.23} it follows that $v, \theta \in L_{\infty }\left(\Omega^{t}\right)$
$\nb v, \nb \t  \in L_{q}\left(\Omega^{t}\right)$ for any finite $q$. Then
\begin{equation}\lb{7.24}
    \begin{aligned}
    &|\nabla(v'\cdot \nb v)|_{5', \Ot} &&\leq c d_{5}^{2},\\
     &|\nabla(v'\cdot \nb \t)|_{5', \Ot} &&\leq c d_{5}^{2},\\
& \left|\partial_{t}^{1 / 2}(v' \cdot \nabla v)\right|_{5, \Omega t} &&\leq c d_{5}^{2},\\
& \left|\partial_{t}^{1 / 2}(v' \cdot \nabla \t)\right|_{5, \Omega t}&& \leq c d_{5}^{2},
\end{aligned}
\end{equation}
where $\partial_{t}^{1 / 2}$ means the partial derivative. 
Estimates \rf{7.24} and Lemma \ref{lemma 7.8} imply
\begin{equation}\lb{7.25}
    \begin{aligned}
 & \|v\|_{W_{10 / 3}^{3,3 / 2} (\Ot)}+\|\theta\|_{W_{10 / 3}^{3,3 / 2} (\Ot)}\\
&\leq  c \bigg(\phi ( \theta _ { * },\theta ^ { * } ) (\|\theta\|_{W_{10 / 3}^{1,1 / 2}\left(\Omega^{t}\right)}+\|f\|_{W_{10 / 3}^{1,1 / 2}(\Ot)} )\\
&+ \|g\|_{W_{10 / 3}^{1,1 / 2}(\Ot)}+\left\|v_{0}\right\|_{W_{10 / 3}^{12/5}(\Omega)}+\left\|\theta_{0}\right\|_{W_{10 / 3}^{12 / 5}(\Omega)} + d_{5}^{2}\bigg).
\end{aligned}
\end{equation}

Applying the interpolation to eliminate the norm $\|\theta\|_{W_{10 / 3}^{1,1 / 2}\left(\Omega^{t}\right)}$ from the r.h.s. of \rf{7.25} and using \rf{2.8} we obtain
\begin{multline}\lb{7.26}
 \|v\|_{W_{10 / 3}^{3,3 / 2} (\Ot)}+\|\theta\|_{W_{10 / 3}^{3,3 / 2} (\Ot)} \leq \phi\bigg(D_{0}, \phi\left(\theta_{*}, \theta^{*}\right),\\
\|f\|_{W_{10 / 3}^{1,1 / 2}(\Ot)} ,
\|g\|_{W_{10 / 3}^{1,1 / 2}(\Ot)},\left\|v_{0}\right\|_{W_{10 / 3}^{12/5}(\Omega)},\left\|\theta_{0}\right\|_{W_{10 / 3}^{12 / 5}(\Omega)} ,d_{5}\bigg)\equiv d_6.
\end{multline}

Continuing the considerations yields
\begin{multline}
\|v\|_{W_{2}^{4,2}(\Omega^{t}}+\|\theta\|_{W_{2}^{4} 2}\left(\Omega^{t}\right) \leq \phi\bigg(\|f\|_{W_{2}^{2, 1}(\Omega^{t}}, \\
\|g\|_{W_{2}^{2, 1}(\Omega^{t})},\left\|v_{0}\right\|_{H^{3}(\Omega)}\left\|\theta_{0}\right\|_{H^{3}(\Omega)},d_6,D_0\bigg).
\end{multline}

This ends the prof.
\end{proof}

\section*{Conflict of interest statement}
The authors report there are no competing interests to declare.

\section*{Data availability statement}
The authors report that there is no data associated with this work.

\bibliographystyle{amsplain}

\end{document}